\documentclass[11pt]{amsart}

\usepackage{hyperref}
\usepackage{amsmath,amsthm,amssymb,mathtools,tocvsec2,pdflscape,cite}
\usepackage{latexsym}
\usepackage{color,soul}
\usepackage[utf8]{inputenc}
\usepackage[T1]{fontenc}
\usepackage{tikz}
\usepackage{cancel}

\usepackage{tikz-network,comment}
\usepackage{lscape}

\usepackage[tmargin=1in,bmargin=1in,rmargin=1in,lmargin=1in]{geometry}

\usepackage{amsmath}
\usepackage{graphicx}
\usepackage{enumerate}
\usepackage{tabularx}
\usetikzlibrary{arrows.meta}
\tikzset{>={Latex[width=1.2mm,length=1.7mm]}}

\theoremstyle{plain}
\newtheorem{thm}{Theorem}[section]

\newtheorem{prop}[thm]{Proposition}

\numberwithin{equation}{section}

\newtheorem{lemma}[equation]{Lemma}
\newtheorem{theorem}[equation]{Theorem}

\newtheorem{utheorem}{\textrm{\textbf{Theorem}}}

\newtheorem{ucor}[utheorem]{\textrm{\textbf{Corollary}}}

\theoremstyle{definition}
\newtheorem{ualgo}[utheorem]{\textrm{\textbf{Algorithm}}}
\newtheorem{example}[thm]{Example}
\newtheorem{defn}[thm]{Definition}
\newtheorem{remark}[thm]{Remark}
\newcommand{\diag}{\mathrm{diag}}
\newcommand{\C}{\mathbb{C}}
\newcommand{\R}{\mathbb{R}}

\newcommand{\sn}{\mathrm{S}_n}

\newcommand{\csn}{\mathbb{C}[\sn]}
\newcommand{\I}{\mathrm{I}}
\newcommand{\El}{\mathrm{E}}

\newcommand{\K}{\mathcal{B}}

\newcommand{\imm}{\mathrm{Imm}}
\newcommand{\sumsb}[1]{\sum_{\substack{#1}}} 
\newcommand{\defeq}{:=}
\newcommand{\spn}{\mathrm{span}}
\newcommand{\sgn}{\mathrm{sgn}}
\newcommand{\ntnsp}{\negthinspace}
\newcommand{\permmon}[2]{#1_{1,#2_1} \ntnsp\cdots {#1}_{n,#2_n}}
\newcommand{\tn}{T_n(\xi)}
\newcommand{\So}{\mathfrak{C}}
\newcommand{\Ro}{\mathfrak{R}}
\newcommand{\Shi}{\mathfrak{S}}

\newcommand{\In}{\mathcal{I}}
\newcommand{\A}{\alpha}

\newcommand{\Gr}{\mathrm{Gr}}

\newcommand{\TGr}{\mathrm{Gr}^{\geq 0}(m,m+n)}

\DeclarePairedDelimiter\floor{\lfloor}{\rfloor}

\begin{document}



\title[Chevalley operations on TNN Grassmannians]{Chevalley operations on TNN Grassmannians}

\author[Prateek Kumar Vishwakarma]{Prateek Kumar Vishwakarma \\ Universit\'e Laval}

\address[]{D\'epartement de math\'ematiques et de statistique, Universit\'e Laval, Qu\'ebec, Canada}
\email{prateek-kumar.vishwakarma.1@ulaval.ca,~prateekv@alum.iisc.ac.in}

\date{\today}



\bibliographystyle{dart}

\begin{abstract}
Lusztig showed that invertible totally nonnegative (TNN) matrices form a semigroup generated by positive definite diagonal matrices and Chevalley generators. From the Grassmann analogue of these, we introduce \textit{Chevalley operations} on index sets, which we show have a rich variety of applications. We begin by providing a complete classification of all inequalities that are quadratic in Pl\"ucker coordinates over the totally nonnegative part of the Grassmannian:
\[
\sum_{I,J} c_{I,J} \Delta_I \Delta_J \geq 0 \quad\mbox{over}\quad \mathrm{Gr}^{\geq 0} (m, m + n)
\]
where each $c_{I,J}$ is real, and $\Delta_I, \Delta_J$ are Pl\"ucker coordinates satisfying a homogeneity condition. Using an idea of Gekhtman--Shapiro--Vainshtein, we also explain how our Chevalley operations can be motivated from cluster mutations, and lead to working in Grassmannians of smaller dimension, akin to cluster algebras.

We then present several applications of Chevalley operations. First, we obtain certificates for the above inequalities via sums of coefficients $c_{I,J}$ over 321-avoiding permutations and involutions; we believe this refined results of Rhoades--Skandera [\textit{Ann.\ Comb.}\ 2005] for TNN-matrix inequalities via their Temperley--Lieb immanant idea.

Second, we provide a novel proof via Chevalley operations of Lam's log-supermodularity of Pl\"ucker coordinates [\textit{Current Develop.\ Math.}\ 2014]. This has several consequences:
$(a)$~Each positroid, corresponding to the positroid cells in Postnikov's decomposition (2006) of the TNN Grassmannian, is a distributive lattice.
$(b)$~It also yields numerical positivity in the main result of Lam--Postnikov--Pylyavskyy [\textit{Amer.\ J.\ Math.}\ 2007].
$(c)$~We show the coordinatewise monotonicity of ratios of Schur polynomials, first proved by Khare--Tao [\textit{Amer.\ J.\ Math.}\ 2021] and which is the key result they use to obtain quantitative estimates for entrywise transforms of correlation matrices.

Third, we employ our Chevalley operations to show that the majorization order over partitions implicates a partial order for induced character immanants over TNN matrices, proved originally by Skandera--Soskin [\textit{Linear Multilinear Algebra} 2025].

\end{abstract}

\maketitle
\tableofcontents

\section{Introduction}
A real matrix is called \textit{totally nonnegative/positive (TNN/TP)} if the determinant of every square submatrix is nonnegative/positive. These matrices arise in applied and pure mathematics. To give an idea of their prevalence: these matrices occur in work of Gantmacher, Krein, Schoenberg in analysis and interacting particle systems \cite{ASW,GantKreinOsc}; Efron, Karlin in probability and statistics \cite{Karlin64}; Motzkin, Whitney, Cryer, Fallat, Johnson, Pinkus in matrix theory \cite{Motzkin,W,Fallat,Cryer,KarlinPinkus}; Lusztig, Rietsch in representation theory \cite{LusztigTP, Rietsch}; Berenstein, Fomin, Zelevinsky in cluster algebras \cite{BFZ99,berenstein1997total}; Postnikov, Lam in combinatorics \cite{postnikov2006total,lam2015totally,lam2014amplituhedron}; and many others. Some of the recent works to which one may refer are \cite{Brosowsky-Chepuri-Mason-JCTA, Chepuri-Sherman-Bennett-CJM, Khare-Chou-Kannan, FV, FarPost,Galashin-Karp-Lam,Galashin-Pylyavskyy,AK-book,lam2014amplituhedron,lam2015totally,OPostSpey2015,Parisi-Sherman-Bennett-Williams,postnikov2006total,RSkanTLImmp,
JScott2006,Serhiyenko-Sherman-Bennett-Williams,skandera2004inequalities,skan2022sosk,SV,LW}. 

Suppose $n\geq 1$ is an integer. A real polynomial $p({\mathrm{\bf x}})$ in $n^2$ variables $(x_{ij})_{i,j=1}^n=:{\mathrm{\bf x}}$ is called totally nonnegative if $p({\mathrm{\bf x}})\geq 0$ whenever ${\mathrm{\bf x}}$ is totally nonnegative. Lusztig \cite{LusztigTP} extended the notion of total positivity to reductive groups $G$, where the totally nonnegative part, denoted by $G^{\geq 0}$ is a semialgebraic subset generated by \textit{Chevalley generators}. Lusztig also showed \cite{lusztig1998introduction} that $G^{\geq 0}$ is the subset of $G$ where the \textit{dual canonical basis} consists entirely of the TNN polynomials. While this collection has no complete description yet, several recent results have led to substantial progress in (joint) works of Fallat, Gekhtman, Johnson, Rhoades, Skandera, Lam, and others \cite{Fallat-Gekhtman-Johnson, RSkanTLImmp, skandera2004inequalities, FV, lam2015totally}. One of the most notable is the work of Rhoades--Skandera \cite{RSkanTLImmp} characterizing all quadratic homogeneous TNN polynomials via a fundamental collection called Temperley--Lieb immanants. Lam \cite{lam2015totally} delineated the Grassmann analog of these via certain \textit{partial nonncrossing partitions}, and applied it to show that positroids (corresponding to each positroid cell in Postnikov's decomposition) form a distributive lattice. They have (beautifully!) been applied to show that the \textit{majorization order} over \textit{partitions} induces a partial order over the \textit{induced character immanants} for TNN matrices \cite{skan2022sosk}. In a work by Soskin--Vishwakarma \cite{SV} (following Fallat--Vishwakarma \cite{FV}), concepts that are fundamental in studying the geometric, combinatorial, and cluster algebraic properties of the Grassmannian -- \textit{Pl\"ucker relations} and \textit{weakly separated sets} -- are connected with and via the Temperley--Lieb immanants. These immanants have also inspired several other (different!) lines of research and contributed applications. For instance, see works of Farber--Postnikov \cite{FarPost}, Lu--Ren--Shen--Wang \cite{Lu}, Chepuri--Sherman-Bennett \cite{Chepuri-Sherman-Bennett-CJM}, and Nguyen--Pylyavskyy \cite{NguyenPylyavskyy}.

In this project, we revisit the theorem of Loewner, Whitney, Lusztig, and Berenstein--Fomin--Zelevinsky \cite{Loewner55,W,LusztigTP,BFZ99}: invertible TNN matrices are products of positive definite diagonal matrices and Chevalley generators. And inspired by it, we introduce \textit{Chevalley operations} and show several applications. In the first we show that calculated applications of these operations classify all homogeneous, quadratic polynomials that are TNN. In fact, we do this in a general setting, considering all real polynomials that are homogeneous and quadratic in Pl\"ucker coordinates, and classify all which are nonnegative over the entire totally nonnegative part of the Grassmannian. Considering that these classifications can also be obtained via Temperley--Lieb immanants (see Soskin--Vishwakarma \cite{SV} for the reformulation), and the extensive applications of Temperley--Lieb immanants, we further provide several other applications of Chevalley operations (Section~\ref{ChevalleyOp}). 

Namely, we first show the connection between Chevalley operations and Temperley--Lieb immanants via certain certificates using 321-avoiding permutations and 321-avoiding involutions; we believe this refines the work of Rhoades--Skandera \cite{RSkanTLImmp} (Subsection~\ref{TLimm}). Second, we provide an alternate proof of Lam's log-supermodularity of Pl\"ucker coordinates \cite{lam2015totally}. This yields several consequences: $(a)$~Positroids, corresponding to the positroid cells in Postnikov's decomposition \cite{postnikov2006total} of the TNN Grassmannian, form a distributive lattice under $(\min,\max)$ operations. $(b)$~The log-supermodularity also yields a proof of numerical positivity in the main result of Lam--Postnikov--Pylyavskyy \cite{LPP}. $(c)$~We obtain the coordinatewise monotonicity of ratios of Schur polynomials, which was shown by Khare--Tao \cite{KT}; this is the key result that they use to derive quantitative estimates for entrywise transforms of correlation matrices (Subsection~\ref{Lam-sup-mod}).   
Third, this is followed by an alternate proof of the fact that majorization order over partitions yields a partial order over the induced character immanants over TNN matrices, also known as the Barrett--Johnson inequalities for TNN matrices \cite{skan2022sosk} (Subsection~\ref{BJsection}).

In each of these applications, the known proofs are combinatorial via the Temperley--Lieb immanant idea, while the ones we present use the structure of Chevalley operations. In fact, we demonstrate that this structure bears resemblance with certain cluster mutations, which also lead to working in a Grassmannian of smaller dimension, parallel to cluster algebras (Section~\ref{clusteralgebra}). Finally in Section~\ref{WThm_TNNGrass} we prove our main results stated in Section~\ref{ChevalleyOp}, and discuss in Section \ref{futurework} some natural questions that follow as future work.


\subsubsection*{Some standard notations and definitions employed henceforth} In this article we assume that $1\leq m\leq n,$ and define $[m,n]:=\{m,\dots,n\},$ and $[m]:=[1,m].$ The Grassmannian $\Gr{(m,m+n)}$ is the manifold of $m$-dimensional vector subspaces of $\R^{m+n}.$ Every element in $\Gr{(m,m+n)}$ identifies with a full rank $(m+n)\times m$ real matrix $\mathcal{A}$ modulo elementary column operations. Suppose $I\subset [m+n]$ is an $m$-element subset of $[m+n],$ and define $\mathcal{A}_{I,[m]}$ to be the submatrix of $\mathcal{A}$ corresponding to row and column index sets $I$ and $[m],$ respectively. The determinant of these submatrices of $\mathcal{A}$ are given by $\Delta_I(\mathcal{A}):=\det \mathcal{A}_{I,[m]},$ where $(\Delta_I(\mathcal{A}))$ forms the projective coordinates of $V$ and are called the Pl\"ucker coordinates. The totally nonnegative (TNN) Grassmannian $\TGr\subset \Gr{(m,m+n)}$ refers to vector subspaces with all Pl\"ucker coordinates $\Delta_I\geq 0$ for some matrix representative.

\section{Motivation from cluster algebras}\label{clusteralgebra}

In this work we aim to classify all homogeneous determinantal inequalities that are quadratic in Pl\"ucker coordinates over the TNN Grassmannian, via a novel set of operations which we (define later and) call the Chevalley operations. The fundamental idea behind these operations can be traced into the cluster algebra structure on the Grassmannian. Gekhtman--Shapiro--Vainshtein \cite{GekShaVai03} obtained the cluster algebra structure of $\Gr(m,m+n)$ via the standard Poisson structure. This construction involves specific initial clusters for $\Gr(m,m+n),$ one of which is given by the following Pl\"ucker coordinates. Suppose $\mathcal{K}:=[m]\times [n].$
\begin{align}\label{initial-cluster}
\mathrm{\bf x}&:=\mathrm{\bf x}(m,m+n):=\{x_{0,n+1}:=x_{I_{0,n+1}},~x_{ij}:=x_{I_{ij}}:(i,j)\in \mathcal{K}\},\\
\mbox{where}\quad\mathrm{\bf I}_{(m,n)}&:=\big{\{}I_{0,n+1}:=[m],~I_{ij}:=\big{(}[1,m]\setminus [i-\ell(i,j),i]\big{)}\cup [j+m,j+m+\ell(i,j)]:(i,j)\in \mathcal{K}\big{\}},\nonumber\\
\mbox{and}\quad\ell(i,j)&:=\min(i-1,n-j) \mbox{ for }(i,j)\in \mathcal{K}.\nonumber
\end{align}
In showing that each of these Pl\"ucker coordinates form cluster variables, the authors of \cite{GekShaVai03} construct and apply a specially designed sequence $\Ro$ of \textit{cluster transformations} (also known as \textit{cluster mutations}) to the cluster in \eqref{initial-cluster}. In particular, the application of $\Ro$ to the initial cluster $\eqref{initial-cluster}$ yields the corresponding initial cluster $\mathrm{\bf I}_{(m-1,n-1)}$ for $\Gr(m-1,m+n-2).$ This process enabled the authors of \cite{GekShaVai03} to inductively obtain the cluster algebraic structure of the Grassmannian (via the standard Poisson structure). One can check the required details in \cite{GekShaVai03, gekhtman2010cluster}. Nevertheless, in essence, the sequence $\Ro$ of cluster mutations is equivalent to the operation $\Ro_{(m,m+1)}$ defined over $\mathrm{\bf I}_{(m,n)}$ via:
\begin{align}\label{cluster-trans-1}
&\Ro_{(m,m+1)}\big{(}\mathrm{\bf I}_{(m,n)}\big{)}:=\{\Phi_{(m,m+1)}\big{(}I\setminus\{m,m+1\}\big{)}:I\in \mathrm{\bf I}_{(m,n)}(m,m+1)\},\\
\mbox{where}\quad&\mathrm{\bf I}_{(m,n)}{(m,m+1)}:=\{I\in \mathrm{\bf I}_{(m,n)}:\mbox{ either }m\in I \mbox{ and }m+1\not\in I,\mbox{ or }m+1\in I
\mbox{ and }m\not\in I\},\nonumber\\
\mbox{and}\quad&\Phi_{(m,m+1)}:[m+n]\setminus\{m,m+1\} \to [m+n-2] \mbox{ is the (unique) order preserving map}.\nonumber
\end{align}
It can be seen that $\Ro_{(m,m+1)}\big{(}\mathrm{\bf I}_{(m,n)}\big{)} = \mathrm{\bf I}_{(m-1,m+n-2)},$ which is exactly the initial cluster for $\Gr(m-1,m+n-2)$ in \eqref{initial-cluster}. It may be interesting to note that the pair $(m,m+1)$ is unique to perform this reduction.

In addition to constructing the sequence $\Ro$ of cluster mutations, Gekhtman--Shapiro--Vainshtein provide a sequence $\Shi$ of cluster mutations that \textit{shifts} the indices in a initial cluster up by 1 $\pmod {m+n}.$ They also show that $\Shi$ applied to initial cluster \eqref{initial-cluster} yields another initial cluster. If one starts with the shifted initial cluster, and applies the cluster mutation sequence $\Ro$ to it, then the resultant cluster can also be obtained via an operation $\Ro_{(u,v)},$ for some consecutive integers $u,v\in [m+n] \pmod{m+n},$ similar to the operation $\Ro_{(m,m+1)}$ in \eqref{cluster-trans-1} (by essentially replacing $m$ with $u$ and $m+1$ with $v$ in \eqref{cluster-trans-1}). And this process again yields a initial cluster for $\Gr(m-1,m+n-2).$ One should be able to see that the operations $\Ro_{(u,v)}$ can be defined for any arbitrary collection of $m$-element subsets ${\bf I}$ of $[m+n]$ via a definition similar to \eqref{cluster-trans-1}. With this, let us present how the operations $\Ro_{(u,v)},$ for consecutive $u,v\in [m+n] \pmod{m+n},$ naturally extend over the most famous homogeneous quadratic equations over the Grassmannian called the \textit{Pl\"ucker relations}. (Since the theme of the paper is to discuss homogeneous quadratic inequalities, we consider this as a good starting point.) We formally define $\Ro_{(u,v)},$ for $u,v\in [m+n]$ (not necessarily consecutive), over an arbitrary family of $m$-element subsets. Suppose $\mathcal{I}$ denotes an ordered finite indexing set, and for each $\A\in \In,$ $I_\A \subseteq [m+n]$ is an $m$-element subset. Then 
\begin{align}\label{cluster-trans-2}
&\Ro_{(u,v)}\big{\{} I_{\A}:\A\in \In\big{\}}:=\{\Phi_{(u,v)}\big{(}I_\A \setminus\{u,v\}\big{)}:\A \in \In(u,v)\},\\
\mbox{where}\quad&\In(u,v):=\{\A \in \In:\mbox{ either }u\in I_\A \mbox{ and }v\not\in I_\A,\mbox{ or }v\in I_\A
\mbox{ and }u\not\in I_\A\},\nonumber\\
\mbox{and}\quad&\Phi_{(u,v)}:[m+n]\setminus\{u,v\} \to [m+n-2] \mbox{ is the (unique) order preserving map}.\nonumber
\end{align}
A Pl\"ucker relation can be written for two families $\{I_\A\}_{\A\in \In}$ and $\{J_\A\}_{\A\in \In}$ of $m$-element subsets:
\begin{align*}
    \sum_{\A\in \In} c_{\A} \Delta_{I_{\A}}\Delta_{J_{\A}} = 0 \quad \mbox{over}\quad \Gr(m,m+n)\quad \mbox{for some}\quad c_{\A}\in \{\pm 1\}.
\end{align*}
Since there are two families of Pl\"ucker coordinates involved in the relation above, the notion of ``consecutive'' $u,v$ has to be compatible with both the families and thus needs to be redefined. Fortunately for us, Pl\"ucker relations are homogeneous, i.e., the unions with multiplicities $I_\A\Cup J_\A$ are uniform over $\A\in \In.$ So, we have $I:=(I_\A\cup J_\A)\setminus(I_\A\cap J_\A)$ to choose consecutive integers $u,v$ from. We call $u,v$ \textit{consecutive in} $I$ provided for all integers $w$ strictly between $u$ and $v,$ $w\not\in I.$ Suppose the indexing sets are different, i.e., the families are given by $\{I_\A\}_{\A\in \In}$ and $\{J_\A\}_{\A\in \mathcal{J}}.$ Then it can observed that for all consecutive $u,v\in I,$ 
\begin{align*}
\A\in \In(u,v) \quad\mbox{if and only if}\quad\A\in \mathcal{J}(u,v).    
\end{align*}
This leads one to write the following relation for $m'<m$ and $n'<n$:
\begin{align*}
    \sum_{\A\in \In(u,v)} c_{\A} \Delta_{\Phi_{(u,v)}(I_{\A})}\Delta_{\Phi_{(u,v)}(J_{\A})} = 0 \quad \mbox{over}\quad \Gr(m',m'+n')\quad \mbox{for}\quad c_{\A}\in \{\pm 1\}.
\end{align*}
This obtained relation may be empty for certain choices of $u,v$ (for instance if $\In(u,v)=\emptyset$). It can be seen that for all \textit{carefully chosen} values of $u,v$ the aforementioned relation is a Pl\"ucker relation over a Grassmannian of strictly smaller dimension -- just as in the case of initial clusters above, where only a certain $\Ro_{(u,v)}$ yields the initial cluster for the Grassmannian of the smaller dimension.

A Grassmannian can be obtained from initial clusters (via cluster mutations; \cite{GekShaVai03, gekhtman2010cluster}) and from Pl\"ucker relations (via geometry; folklore). We discussed above that these defining objects can be reduced into ones for the smaller dimensional Grassmannian by applying certain operations. Moreover, the Pl\"ucker relations are homogeneous and quadratic, and the inequalities that we wish to classify in this paper are also homogeneous and quadratic. Therefore, we wonder if performing these operations over these inequalities would also perform a similar reduction and yield a novel classification for homogeneous inequalities that are quadratic in Pl\"ucker coordinates over $\TGr$. We discover that the answer is, yes$!$

\section{Chevalley operations and main results}\label{ChevalleyOp}

We begin with the formal definition of the inequalities that we classify in this paper.

\begin{defn}[Homogeneous quadratic determinantal inequalities over TNN Grassmannians]\label{defn:main:1}
Let $1\leq m \leq n$ be integers, and suppose $\mathcal{I}$ is an ordered finite indexing set. Suppose $I_{\alpha},J_{\alpha} \subset [m+n]$ have $m$ elements such that the multisets $I_{\alpha}\Cup J_{\alpha}$ are equal for all $\alpha \in \mathcal{I}.$ For real $c_{\alpha},$ we say that
\begin{align}\label{quad_ineq}
\sum_{\alpha\in \mathcal{I}} c_{\alpha} \Delta_{I_{\alpha}} \Delta_{J_{\alpha}} \mbox{ is nonnegative } \big{(}\geq 0\big{)}\quad  \mbox{over} \quad \TGr
\end{align}
provided $\sum_{\alpha\in \mathcal{I}} c_{\alpha} \Delta_{I_{\alpha}}(\mathcal{A}) \Delta_{J_{\alpha}}(\mathcal{A}) \geq 0$ for all ${(m+n)\times m}$ real matrices $\mathcal{A}$ whose column space lie in $\TGr$. For brevity, we call \eqref{quad_ineq} as a \textit{quadratic inequality}. To avoid confusion, we call \eqref{quad_ineq} as a \textit{quadratic expression} when we are unsure if $\sum_{\alpha\in \mathcal{I}} c_{\alpha} \Delta_{I_{\alpha}} \Delta_{J_{\alpha}} \geq 0$  over $\TGr.$
\end{defn}

For a structured exposition, we need a refined version of the operations $\Ro_{(u,v)}$ (which we discussed in Section~\ref{clusteralgebra}, and which are strongly reminiscent of cluster mutations; we shed more light on this in Remark~\ref{recall-Ruv}). We call these refined versions as the \textit{Chevalley operations}. We shall see later that these are related precisely to the Chevalley generators of the invertible totally nonnegative matrices, hence the name. 

\begin{defn}[Chevalley operations]\label{defn:main:2}
Let $1\leq m\leq n$ be integers. 
\begin{enumerate}
\item Suppose $I\subset [m+n]$ has $m$-elements. For consecutive $u,v\in [m+n],$ define
\begin{align*}
I(u,v):=
\begin{cases}
\big{(}I\setminus \{u\}\big{)}\cup\{v\} & \mbox{ if } u \in I \mbox{ and }v\not\in I, \mbox{ and} \\
I & \mbox{ otherwise.}
\end{cases}
\end{align*}

\item Let $\mathcal{I}$ be an ordered finite indexing set, and let $I_{\alpha},J_{\alpha}\subset [m+n]$ be $m$ element subsets for $\alpha \in \mathcal{I}.$ For consecutive $u,v\in [m+n],$ define the Chevalley operation $\So_{(u,v)}$ via
\begin{align*}
&\So_{(u,v)}\Big{(}\big{(}I_{\alpha},J_{\alpha}\big{)}:\alpha \in \mathcal{I} \Big{)}  :=  
\Big{(}\big{(}I_{\alpha}(u,v),J_{\alpha}(u,v)\big{)}:\alpha \in \mathcal{I}(u,v) \Big{)},\\
\mbox{where}\quad &\alpha(u,v):= \{X\in \{I_{\alpha},J_{\alpha}\}: u\in X \mbox{ and }v\not\in X\},\\
\mbox{and}\quad &\mathcal{I}(u,v) := \{\alpha \in \mathcal{I}: |\alpha(u,v)| = \max_{\alpha \in \mathcal{I}} |\alpha(u,v)|\}.
\end{align*}
\end{enumerate}
\end{defn}

A (weaker) version of Chevalley operations appeared in a recent work of Fallat--Vishwakarma \cite{FV}, along with a question re: certain recursive nature of determinantal inequalities for totally nonnegative matrices (see \cite[Question~A]{FV}). Using the (stronger) Chevalley operations in Definition~\ref{defn:main:2}, we completely answer that question for quadratic determinantal inequalities over the TNN Grassmannian (and TNN matrices).

\begin{utheorem}[Classification of \eqref{quad_ineq} via Chevalley operations]\label{L:1}
Let $1\leq m \leq n$ be integers; notations in Definitions~\ref{defn:main:1} and \ref{defn:main:2}. The inequality \eqref{quad_ineq} is valid if and only if for all consecutive $u,v\in [m+n]$
\begin{align}\label{L:1:Eqn:2}
\sum_{\alpha\in \mathcal{I}(u,v)} c_{\alpha} \Delta_{I_{\alpha}(u,v)} \Delta_{J_{\alpha}(u,v)} \geq 0 \quad \mbox{over}\quad\TGr.
\end{align}
This is under the convention that the sum and product over the empty set are $0$ and $1$ respectively. 
\end{utheorem}

\begin{remark}
Note that using Theorem~\ref{L:1} in applications is meaningful provided the expression \eqref{L:1:Eqn:2} is different from \eqref{quad_ineq}, i.e., there exists $\alpha \in \In$ such that $(I_{\A},J_{\A})\neq (I_{\A}(u,v),J_{\A}(u,v)).$
\end{remark}

\begin{remark}
The Chevalley operations -- defined over ordered collections of pairs of sets -- naturally extend over to the class of quadratic inequalities \eqref{quad_ineq}. Theorem~\ref{L:1} shows the mechanics of this extension: a quadratic inequality \eqref{quad_ineq} holds if and only if the action of all Chevalley operations on it -- via their action on the corresponding collection of pairs $(I_\A,J_\A)_{\A\in \In}$ -- yields a valid quadratic inequality. More precisely, for all consecutive $u,v\in [m+n],$ over $\TGr$
\begin{align}\label{cheva_op_mech:1}
\sum_{\alpha\in \mathcal{I}} c_{\alpha} \Delta_{I_{\alpha}} \Delta_{J_{\alpha}} \geq 0 ~~ \Longleftrightarrow ~~\So_{(u,v)}\Bigg{(}\sum_{\alpha\in \mathcal{I}} c_{\alpha} \Delta_{I_{\alpha}} \Delta_{J_{\alpha}} \Bigg{)} := \sum_{\alpha\in \mathcal{I}(u,v)} c_{\alpha} \Delta_{I_{\alpha}(u,v)} \Delta_{J_{\alpha}(u,v)} \geq 0.
\end{align}
\end{remark}

This provides a recursive (and efficient) tool to classify quadratic inequalities of the form \eqref{quad_ineq}. Namely, a careful application of the Chevalley operations reduces the problem of classifying quadratic inequalities \eqref{quad_ineq} over a TNN Grassmannian to such classifications over the TNN Grassmannian of a smaller dimension. This reduction requires the following result, which ``simplifies'' a given quadratic inequality.

\begin{utheorem}[Simplification]\label{L:2}
Let $1\leq m \leq n$ be integers; notations as in Definition~\ref{defn:main:1}. Suppose $\eta:=|I_{\alpha}\setminus J_{\alpha}|,$ and let $\Phi: \big{(} I_{\alpha}\setminus J_{\alpha}\big{)} \cup \big{(} J_{\alpha}\setminus I_{\alpha}\big{)} \to [2\eta]$ be the unique order preserving map. Define $K_{\A}:=\Phi\big{(}I_{\A}\setminus J_{\A}\big{)}$ and $L_{\A}:=\Phi\big{(}J_{\A}\setminus I_{\A}\big{)}$ for $\alpha \in \In.$ Then inequality~\eqref{quad_ineq} is valid if and only if 
\begin{align}\label{L:2:Eqn2}
\sum_{\alpha\in \mathcal{I}} c_{\alpha} \Delta_{K_{\alpha}} \Delta_{L_{\alpha}} \geq 0 \quad\mbox{over}\quad \Gr^{\geq 0}(\eta,2\eta).
\end{align}
\end{utheorem}

Theorem~\ref{L:2} removes all the redundacies, like the indices that are common in $I_\A$ and $J_\A$ and hence do not contribute towards the validity of \eqref{quad_ineq}, and yields an equivalent expression over the relevant TNN Grassmannian. One of the main ideas in this paper is a careful and natural combination of Theorems~\ref{L:1} and \ref{L:2}, which we encapsulate as an algorithm.

\begin{ualgo}\label{algo}
Suppose an expression \eqref{quad_ineq} is given that we wish to verify over $\TGr.$
\begin{enumerate}
\item Apply Theorem~\ref{L:2} to remove all the redundancies and obtain a simplified form of the given quadratic expression. This gives an expression of type \eqref{L:2:Eqn2} over $\Gr^{\geq 0}(\eta,2\eta)$.

\item Then apply Theorem~\ref{L:1} for any consecutive $u,v\in [2\eta]$ to obtain a quadratic expression of type \eqref{L:1:Eqn:2} over $\Gr^{\geq 0}(\eta,2\eta).$ In this new expression, $v\in K_\A\cap L_\A$ and $u\not\in K_\A \cup L_\A,$ for all $\A\in \In{(u,v)}.$

\item Next, one applies Theorem~\ref{L:2} to the expression obtained in the previous step. Since -- $v\in K_\A\cap L_\A$ and $u\not\in K_\A \cup L_\A,$ for all $\A\in \In{(u,v)}$ -- this yields an equivalent quadratic expression over the TNN Grassmannian of the smaller dimension $\Gr^{\geq 0}(\eta-1,2\eta -2).$

\item In the final step to prove that the given quadratic inequality \eqref{quad_ineq} is valid, one needs to verify that the quadratic expression obtained in Step (3) is valid over $\Gr^{\geq 0}(\eta-1,2\eta -2),$ for all consecutive $u,v\in [2\eta]$ (chosen in Step (2)).
\end{enumerate}
\end{ualgo}

\begin{remark}[Cluster mutations and Chevalley operations]\label{recall-Ruv}

We believe it would be apparent that Step (2) to Step (4) in Algorithm~\ref{algo} is precisely how $\Ro_{(u,v)}$ acts \textit{together} on $\big{(}K_\A\big{)}_{\A\in \In}$ and $\big{(}L_\A\big{)}_{\A\in \In},$ and hence on the expression \eqref{L:2:Eqn2}, thus demonstrating that the ``application part'' of the Chevalley operations is parallel to the action of the sequence $\Ro$ of cluster mutations over the initial cluster \eqref{initial-cluster}.
\end{remark}

Sometimes the given quadratic inequality \eqref{quad_ineq} for $m=n=\eta$ involves only a certain type of Pl\"ucker coordinates $\Delta_{I_\A}$ and $\Delta_{J_\A}.$ More precisely, we are referring to the ones corresponding to principal minors in totally nonnegative matrices; we define these Pl\"ucker coordinates.

\begin{defn}[Pl\"ucker coordinates for principal minors]\label{defn:main:3}
Suppose $\eta \geq 1$ is an integer. The required Pl\"ucker coordinates (that correspond to the principal minors of $\eta \times \eta$ matrices) are given by the collection: 
\[
\In_{\eta}:=\big{\{}I:=P\cup \big{(} [\eta+1,2\eta]\setminus (2\eta+1-P) \big{)}: P\subseteq [\eta] \big{\}}.
\]
To address the required composition of the Chevalley operations (in the next result) we also define:
\[
u^*:=2\eta+1-u \quad \mbox{for all}\quad u\in [2\eta]. 
\]
(We shall later see that these elements of $\mathcal{I}_\eta$ correspond precisely to the determinantal inequalities over $\eta \times \eta$ totally nonnegative matrices involving only the \textit{principal submatrices} -- see \eqref{Plucker-embed}.)
\end{defn}

In light of Algorithm~\ref{algo}, an immediate question is to refine it for quadratic inequalities over $\Gr^{\geq 0}(\eta,2\eta)$ that involve Pl\"ucker coordinates from $\In_\eta$ so that the inequality obtained in Step (3) of Algorithm~\ref{algo} involves Pl\"ucker coordinates from $\In_{\eta'}$ for $\eta'<\eta.$ Our next result (with Theorem~\ref{L:2}) yields this refinement.

\begin{utheorem}\label{L:3}
Let $\eta \geq 1$ be an integer, and suppose $\mathcal{I}$ is an ordered finite indexing set. Suppose $I_{\alpha},J_{\alpha}\in \In_\eta$ are such that the multisets $I_\alpha \Cup J_\alpha$ are equal for all $\alpha \in \In$. Then for $c_{\alpha}\in \R,$
\begin{align}\label{L:3:Eqn:1}
\sum_{\alpha\in \mathcal{I}} c_{\alpha} \Delta_{I_{\alpha}} \Delta_{J_{\alpha}} \geq 0 \quad\mbox{over}\quad \Gr^{\geq 0}(\eta,2\eta)
\end{align}
if and only if the composite operation $\So_{(v^*,u^*)} \circ \So_{(u,v)}$ applied on \eqref{L:3:Eqn:1} via \eqref{cheva_op_mech:1} yields a valid inequality for all consecutive $u,v\in [\eta+1].$ That is, if $\mathcal{J}:=\big{(}\mathcal{I}(u,v)\big{)}(v^*,u^*),$ $K_\A:=\big{(}I_\A(u,v)\big{)}(v^*,u^*),$ and $L_\A:=\big{(}J_\A(u,v)\big{)}(v^*,u^*),$ for all $\A\in \mathcal{J},$ then \eqref{L:3:Eqn:1} is valid if and only if 
\begin{align}\label{L:3:Eqn:2}
\sum_{\alpha\in \mathcal{J}} c_{\alpha} \Delta_{K_{\alpha}} \Delta_{L_{\alpha}} \geq 0 \quad\mbox{over} \quad \Gr^{\geq 0}(\eta,2\eta).
\end{align}
Moreover, $\mathcal{J}=\mathcal{I}(u,v),$ and $K_\A,L_\A \in \In_{\eta}$ for all $\A\in \mathcal{J}.$
\end{utheorem}

\begin{remark}
Theorem~\ref{L:3} asserts that an expression of the form \eqref{L:3:Eqn:1} is valid if and only if the action of composite Chevalley operations $\So_{(v^*,u^*)}\circ \So_{(u,v)}$ over it yields valid inequalities. More precisely, $\sum_{\alpha\in \mathcal{I}} c_{\alpha} \Delta_{I_{\alpha}} \Delta_{J_{\alpha}} \geq 0$ over $\Gr^{\geq 0}(\eta,2\eta)$ if and only if
\begin{align*}
\So_{(v^*,u^*)}\circ \So_{(u,v)}\Bigg{(}\sum_{\alpha\in \mathcal{I}} c_{\alpha} \Delta_{I_{\alpha}} \Delta_{J_{\alpha}} \Bigg{)} = \sum_{\alpha\in 
\big{(}\mathcal{I}(u,v)\big{)}(v^*,u^*)} c_{\alpha} \Delta_{\big{(}I_{\alpha}(u,v)\big{)}(v^*,u^*)} \Delta_{\big{(}J_{\alpha}(u,v)\big{)}(v^*,u^*)} \geq 0
\end{align*}
over $\Gr^{\geq 0}(\eta,2\eta),$ for all consecutive $u,v\in [\eta+1].$ One may notice a striking difference between Theorem~\ref{L:1} (for $m=n=\eta$) and Theorem~\ref{L:3}. The $m=n=\eta$ version of Theorem~\ref{L:1} states that there are exactly $2\eta - 1$ quadratic expressions (corresponding to each consecutive pair $u,v\in [2\eta]$) that one needs to verify to validate the given quadratic expression. However, if the given quadratic expression involves only the Pl\"ucker coordinates corresponding to the principal minors (Definition~\ref{defn:main:3}) then Theorem~\ref{L:3} states that there are only $\eta+1$ quadratic expressions to verify for the validation. This results in the following refinement of Algorithm~\ref{algo}.
\end{remark}

\begin{ualgo}\label{algo1}
Suppose the quadratic expression over $\Gr^{\geq}(\eta,2\eta)$ obtained in Step (1) of Algorithm~\ref{algo} only involves Pl\"ucker coordinates from $\In_{\eta}.$
\begin{enumerate}
    \item[($2'$)] Then apply Theorem~\ref{L:3} for any consecutive $u,v\in [\eta+1]$ to obtain a quadratic expression \eqref{L:3:Eqn:2}. In this new expression, $v,u^*\in K_\A\cap L_\A$ and $u,v^*\not\in K_\A \cup L_\A,$ for all $\A\in \In{(u,v)}.$
    \item[($3'$)] Since $v,u^*\in K_\A\cap L_\A$ and $u,v^*\not\in K_\A \cup L_\A,$ for all $\A\in \In{(u,v)}$ in the obtained expression, applying Theorem~\ref{L:2} over it yields a quadratic expression either over 
    \begin{align*}
\begin{cases}
    \Gr^{\geq 0}(\eta-1,2\eta-2) \mbox{ with all Pl\"ucker coordinates from $\In_{\eta-1}$} & \mbox{ if }\{u,v\}=\{\eta,\eta+1\}, \mbox{ and}\\
    \Gr^{\geq 0}(\eta-2,2\eta-4) \mbox{ with all Pl\"ucker coordinates from $\In_{\eta-2}$} & \mbox{ otherwise.}
\end{cases}
    \end{align*}
    \item[($4'$)] And finally, similar to the last step in Algorithm~\ref{algo}, one needs to verify if the quadratic expression obtained in step ($3'$) is valid, but this time, for all consecutive $u,v\in [\eta+1],$ unlike the general case in Algorithm~\ref{algo} where one had to run over all consecutive $u,v\in [2\eta].$
\end{enumerate}
\end{ualgo}

Theorem~\ref{L:3} (and so Algorithm~\ref{algo1}) leads to a novel result for quadratic inequalities of type \eqref{L:3:Eqn:1}. This involves a certificate via nonnegativity of sums of coefficients $c_\A,$ $\A\in \In,$ corresponding to the \textit{321-avoiding involutions} in $\mathrm{S}_\eta$ (which can be identified with a subclass of the \textit{Temperley--Lieb immanants}). We shall discuss this in Subsection~\ref{TLimm}. As we mentioned earlier, this is followed by two other applications: a new proof of the log-supermodularity of Pl\"ucker coordinates, and a novel proof of the Barrett--Johnson inequality for TNN matrices. All of these applications are based on the structure of the Chevalley operations, a reminiscent of cluster mutations.

\section{Applications of Chevalley operations}\label{applications}

In this section we demonstrate connections and applicability of Chevalley operations with other well-known notions and results in total positivity, beginning with the Temperley--Lieb immanants.

\subsection{Sums over $321$-avoiding permutations \& involutions}\label{TLimm}

The classification of quadratic inequalities \eqref{quad_ineq} can also be obtained via the Temperley--Lieb immanants idea of Rhoades and Skandera \cite{RSkanTLImmp} -- see Soskin--Vishwakarma \cite{SV} for the required reformulation. In this section we present the connection between Temperley--Lieb immanants and our classification of \eqref{quad_ineq} via Chevalley operations. We recall Temperley--Lieb immanants and how they are fundamental in classifying quadratic inequalities \eqref{quad_ineq}.

Recall that a real polynomial $p({\bf x})$ in matrix entries ${\bf x}=(x_{ij})$ is called \emph{totally nonnegative (TNN)} provided $p({\bf x})\geq 0$ whenever ${\bf x}=(x_{ij})$ is a totally nonnegative matrix. Littlewood \cite{LittlewoodTGC} and Stanley \cite{StanPos} introduced the notion of immanants: suppose $f: \sn \rightarrow \mathbb C,$ and define \emph{$f$-immanant} to be the polynomial
\begin{equation}\label{eq:immdef}
\imm_f({\bf x}) \defeq \sum_{w \in \sn} f(w) \permmon xw \in \mathbb C[{\bf x}].
\end{equation}
Fix $\xi\in \C$; consider the Temperley--Lieb algebra $T_n(\xi)$ over $\C$ generated by $t_1,\dotsc,t_{n-1}$ subject to:
\begin{alignat*}{2}
t_i^2 &= \xi t_i, &\qquad &\text{for } i=1,\dotsc,n-1, \\
t_i t_j t_i &= t_i,   &\qquad &\text{if }  |i-j|=1,\\
t_i t_j &= t_j t_i,   &\qquad &\text{if }  |i-j| \geq 2.
\end{alignat*}
(This is also defined as the quotient of the Hecke algebra $H_n(q).$) The ``monoid'' $\K_n$ generated by $t_1,\dots,t_{n-1}$ subject to the aforementioned relations forms the standard basis of $T_n(\xi).$ One of the ways to see these basis elements is to look at all 321-avoiding permutation in $\sn$ and replace each $s_i\leftrightarrow t_i.$ This identification leads to their identification with Kauffman diagrams. A \textit{Kauffman diagram} is a matching on a $2n$-cycle such that the edges do not intersect and lie inside the convex hull generated by the vertices of the $2n$-cycle.

Now consider the isomorphism $T_n(2) \cong \csn/(1 + s_1 + s_2 + s_1s_2 + s_2s_1 + s_1s_2s_1)$ \cite{FanMon,GHJ,WestburyTL} via
\begin{equation}\label{eq:sntotn}
\begin{aligned}
\sigma : \csn \rightarrow \tn ~\mbox{ with }~ s_i \xmapsto{\sigma} t_i - 1.
\end{aligned}
\end{equation}
Define the following function corresponding to each basis element in $\tau \in \K_n,$
\begin{equation}\label{eq:ftau}
\begin{aligned}
f_\tau: \sn \rightarrow \mathbb{C} ~\mbox{ with }~ w \mapsto \text{ coefficient of $\tau$ in } \sigma(w),
\end{aligned}
\end{equation}
and extend it linearly over $\csn$. Finally, define the Temperley--Lieb immanants as 
\begin{equation*}
\imm_{\tau}({\bf x}):=\imm_{f_\tau}({\bf x}) 
= \sum_{w \in \sn} f_{\tau}(w){x}_{1,w_1} \cdots {x}_{n,w_n}, \quad \mbox{where}\quad {\bf x}=(x_{ij})_{i,j=1}^{n}.
\end{equation*}
Rhoades--Skandera \cite{RSkanTLImmp} showed that Temperley--Lieb immanants are a basis of the space 
\begin{equation}\label{eq:tlspace}
\spn_{\mathbb R} \{ \det {\bf x}_{P,Q}  \det {\bf x}_{P^c,Q^c}  \,|\, P,Q \subseteq [n] \mbox{ with }|P|=|Q| \}
\end{equation} 
and that they are TNN. In fact, these are the extreme rays of the cone of TNN immanants.
\begin{thm}[Rhoades--Skandera \cite{RSkanTLImmp}]\label{SB}
Given a function $f:\sn \to \R,$ the immanant
\begin{equation}\label{eq:sumofprodsof2minors}
\imm_f({\bf x}) = \sumsb{P,Q \subseteq [n]\\ |P|=|Q|} c_{P,Q}
\det {\bf x}_{P,Q} \det {\bf x}_{P^c, Q^c} 
\end{equation}
is TNN if and only if it is a nonnegative linear combination of Temperley--Lieb immanants. Moreover, each $\det {\bf x}_{P,Q} \det {\bf x}_{P^c, Q^c}$ is the sum of Temperley--Lieb immanants that correspond to Kauffman diagrams with edges connecting elements of $P\cup Q^{c}$ and $P^c\cup Q$ on the $2n$-cycle.
\end{thm}
Theorem~\ref{SB} can be applied in identifying quadratic inequalities \eqref{quad_ineq} when $m=n$ and $I_\A\cap J_\A=\emptyset$. Recall that the TNN matrices sit inside in the TNN Grassmannian:
\begin{align}\label{Plucker-embed}
\{\mbox{all $n\times m$ TNN matrices}\}\hookrightarrow \Gr^{\geq 0}(m,m+n)\quad\mbox{where}\quad A \mapsto \overline{A}:= \begin{pmatrix} A \\ W_{0} \end{pmatrix}
\end{align}
with $W_{0}:=(w_{ij})=\big((-1)^{i+1}\cdot\delta_{j,m-i+1}\big)^{m}_{i,j=1},$ i.e. $w_{ij}=(-1)^{i+1}$ if $j=m-i+1,$ and $0$ otherwise. This yields a one-to-one correspondence between the minors of $A$ and the maximal minors of $\overline{A}$ via 
$
\det A_{P,Q}=\det \overline{A}_{I,[m]}=\Delta_{I}(\overline{A})
$
where $I:= P \cup \{ m+n+1-j \,|\, j \in [m] \setminus Q\}$ for all $P\subseteq [n],Q \subseteq [m]$ with $|P|=|Q|.$
This along with the projective geometry of the Grassmannian provides us with the following equivalence between inequalities that are quadratic in minors of TNN matrices and inequalities that are quadratic in Pl\"ucker coordinates over the TNN Grassmannian.

\begin{theorem}[Soskin--Vishwakarma \cite{SV}]\label{Ineq-equivalence}
Let $1\leq m\leq n$ be integers, and suppose sets $P_i\subseteq [n]$ and $Q_i\subseteq [m]$ with $|P_i|=|Q_i|,$ for $i=1,2.$ Let $c_{P_1,Q_1,P_2,Q_2}=c_{I,J}\in \R,$ where each 
\begin{align*}
I=P_1\cup \{ m+n+1-j \,|\, j \in [m] \setminus Q_1\}\quad \mbox{and}\quad J=P_2\cup \{ m+n+1-j \,|\, j \in [m] \setminus Q_2\}.
\end{align*}
Then the following three inequalities are equivalent: 
\begin{align*}
\sum_{P_1,Q_1,P_2,Q_2} c_{P_1,Q_1,P_2,Q_2}\det A_{P_1,Q_1} \det A_{P_2,Q_2} & \geq 0\quad \forall A_{n\times m} \quad TNN. \\
\sum_{I,J} c_{I,J}\Delta_{I}(\overline{A}) \Delta_{J}(\overline{A}) &\geq 0 \quad \forall A_{n\times m} \quad TNN. \\
\sum_{I,J} c_{I,J}\Delta_{I} \Delta_{J} &\geq 0\quad \mbox{ over } \quad \Gr^{\geq 0}(m,m+n).
\end{align*}
\end{theorem}

The correspondence between the inequalities in Theorem~\ref{Ineq-equivalence} is compatible with the Temperley--Lieb immanant idea.

\begin{thm}[Rhoades--Skandera \cite{RSkanTLImmp}, Soskin--Vishwakarma \cite{SV}]\label{t:oneprod}
Suppose $I$ runs over $n$-element subsets of $[2n],$ and $c_I\in \R.$ Then
\begin{align}\label{eq:pl0}
\sum_{I} c_{I}\Delta_{I} \Delta_{I^{\mathrm{c}}} \geq 0 \quad\mbox{over}\quad\Gr^{\geq 0}(n,2n)
\end{align}
if and only if $\sum_{I} c_{I}\Delta_{I}(\overline{{\bf x}}) \Delta_{I^{\mathrm{c}}}(\overline{{\bf x}})$ is a nonnegative linear combination of Temperley--Lieb immanants, where ${\bf x}=(x_{ij})_{i,j=1}^{n}$ and $\overline{{\bf x}}$ is as in \eqref{Plucker-embed}. Moreover, we have for each $I$ that 
\begin{equation}\label{eq:pl1}
\Delta_{I}(\overline{{\bf x}})\Delta_{I^c}(\overline{{\bf x}}) = \sum_{\tau \in \K_n}b_{\tau}\imm_{\tau}({\bf x}),    
\end{equation}
where $b_{\tau}=1$ if each edge in the Kauffman diagram of $\tau$ connects elements from $I$ and $I^c$, and $0$ otherwise.
\end{thm}
Starting from our classification in Theorem~\ref{L:1} and Theorem~\ref{L:2}, we obtain a certificate for \eqref{quad_ineq} in terms of the sums of coefficients $(c_\A)_{\A\in\In}$ over \textit{321-avoiding permutations}; which identify with the Temperley--Lieb immanants.

\begin{defn}[321-avoiding permutations]
Suppose $\eta\geq 1.$ A permutation $\omega\in\mathrm{S}_{\eta}.$ $\omega\in \mathrm{S}_{\eta}$ is called 321-avoiding if there do not exist $i<j<k \in [\eta]$ such that $\omega(k)<\omega(j)<\omega(i).$
\end{defn}

It is well known that the 321-avoiding permutations are identified with \textit{Dyck words}. A Dyck word of semilength $\eta$ is a word in $\eta$ occurrences of the character \textbf{(} and $\eta$ occurrences of the character \textbf{)} such that at each point in it the number of occurrences of \textbf{(} is weakly more than that of \textbf{)}. Dyck words naturally define certain collection of partitions of $[2\eta]$ via consecutive integers.

\begin{defn}[Partitions compatible with 321-avoiding permutations]\label{conse_int:noncross_part}

Suppose $\eta\geq 1$ is an integer.
\begin{enumerate}
\item Let $I$ be a set of integers. We call distinct $u,v \in I$ \textit{consecutive in} $I$ if for all $w\in I$ such that either $u\leq w\leq v$ or $v \leq w\leq u$ then either $w=u$ or $w=v.$

\item Let $I$ be a set of $2\eta$ integers. Consider partitions $P$ of $I$ that are defined via
\begin{align*}
P:=\big{\{}\{u_j,v_j\}:u_j, v_j\in I \mbox{ for }j\in [\eta] \big{\}}
\end{align*}
provided $u_{k},v_{k}$ are consecutive in $I \setminus \sqcup_{j=1}^{k-1}\{u_{j},v_{j}\}$ for all $k\in [n],$ and $\sqcup_{j=1}^{\eta}\{u_j,v_j\} = I,$ where we define $\sqcup_{j=1}^{0}\{u_{j},v_{j}\} :=\emptyset.$ Considering the well-known bijection between 321-avoiding permutations and Dyck words, it is reasonable to call the collection of all such partitions $P$ as \textit{partitions compatible with 321-avoiding permutations}.
\end{enumerate}
\end{defn}

We can also infer now that the number of such partitions $P$ compatible with 321-avoiding permutations is the $\eta$-th Catalan number $C_\eta:=\frac{1}{\eta+1}\binom{2\eta}{\eta}.$

\begin{defn}[For sums over 321-avoiding permutations]\label{for_sums_o_noncross}
Let $1\leq m\leq n$ be integers and suppose $\In$ is a finite indexing set. Let $I_\A,J_\A \subset [m+n]$ be $m$ element subsets such that the multisets $I_{\A}\Cup J_\A$ are equal, for all $\A\in \In$. Suppose $I:=\big{(}I_\A \cup J_\A\big{)} \setminus \big{(}I_\A \cap J_\A\big{)},$ and let $P:=\big{\{}\{u_j,v_j\}:u_j, v_j\in I \mbox{ for }j=1,\dots \eta\big{\}}$ denote a partition of $I$ compatible with a 321-avoiding permutation of $\mathrm{S}_\eta,$ where $2\eta:=|I|.$ Define
\begin{align*}
\In(P):=\big{\{}\A\in \In : \{u_j,v_j\} \cap I_\A \neq \emptyset \mbox{ and } \{u_j,v_j\} \cap J_\A \neq \emptyset, \mbox{ for all }j\in [\eta] \big{\}}.
\end{align*}
\end{defn}

Now we are ready for the first application of Theorem~\ref{L:1}. This involves a collection of sums over 321-avoiding permutations (which identify with the Temperley--Lieb immanants) and includes a classification obtained for homogeneous quadratic inequalities over totally nonnegative matrices due to Rhoades--Skandera \cite{RSkanTLImmp}.

\begin{utheorem}[Classification via sums over 321-avoiding permutations]\label{Main-thm}
Let $1\leq m \leq n$ be integers; notation as in Definitions~\ref{defn:main:1}, \ref{conse_int:noncross_part}, and \ref{for_sums_o_noncross}. The quadratic inequality~\eqref{quad_ineq} is valid if and only if
\begin{align}\label{Eqn1:Main-thm}
\sum_{\A\in \In(P)} c_{\A} \geq 0\quad \mbox{ for all partitions $P$ of $I$ compatible with a 321-avoiding permutation in $\mathrm{S}_\eta$}.  
\end{align}
\end{utheorem}

\noindent Theorem~\ref{Main-thm} gives the certificate for \eqref{quad_ineq} via exactly $C_{\eta}=\frac{1}{\eta+1}\binom{2\eta}{\eta}$ many sums of coefficients $\big{(}c_\A\big{)}_{\A\in \In}.$

The proof given by Rhoades--Skandera \cite{RSkanTLImmp} involves the construction of totally nonnegative matrices for the wiring diagrams of each 321-avoiding permutation, which also refer to the Temperley--Lieb immanants. Our proof is an application of Chevalley operations by Theorem~\ref{L:1}. We introduce some notation:

\begin{defn}[Certain composite Chevalley operations]\label{compositive-Chev-Op-1}
Let $1\leq m\leq n$ be integers; notations in Definitions~\ref{defn:main:1} and \ref{defn:main:2}. We define a compositive Chevalley operation that depends on $\mathcal{P}:=(I_\A,J_\A)_{\A\in \In},$ so we denote it by $\So^{\mathcal{P}}_{(-,-)},$ and define it in the following steps:
\begin{enumerate}
    \item Suppose $u\neq v \in [m+n],$ and define
\begin{align*}
\overline{\So}_{(u,v)}:=
\begin{cases}
\So_{(v-1,v)}\circ \dots \circ \So_{(u+1,u+2)} \circ \So_{(u,u+1)} & \mbox{ if } u<v,\\
\So_{(v+1,v)}\circ \dots \circ \So_{(u-1,u-2)} \circ \So_{(u,u-1)} & \mbox{ if } u>v.
\end{cases}
\end{align*}
    \item Now suppose $u<v$ are consecutive in $I:=(I_\A \cup J_\A)\setminus (I_\A \cap J_\A),$ and consider all indices $u<r_1<r_2<\dots<r_k<v$ such that each $r_j \in I_\A \cap J_\A.$ Define the following composition of Chevalley operations:
\begin{align*}
\So_{(v,u)}^{\mathcal{P}}:=\So_{(u,v)}^{\mathcal{P}}:=
\begin{cases}
\overline{\So}_{(u,v)} & \mbox{ if }\{r_j\}=\emptyset,\mbox{ and}\\
\overline{\So}_{(u+k,v)}\circ\overline{\So}_{(r_k,u+k-1)}\circ\dots\circ\overline{\So}_{(r_2,u+1)} \circ \overline{\So}_{(r_1,u)} & \mbox{ otherwise}.
\end{cases}
\end{align*}
\end{enumerate}
\end{defn}

We are now ready for the next proof.

\begin{proof}[Proof of Theorem~\ref{Main-thm}]
Suppose $\mathcal{P}_1:= (I_\A,J_\A)_{\A\in\In}.$ One can see that for a consecutive pair $u,v$ in $(I_\A \cup J_\A)\setminus (I_\A \cap J_\A),$ the application of ${\So}_{(u,v)}^{\mathcal{P}_1}$ on \eqref{quad_ineq} yields an inequality of the form \eqref{quad_ineq}, say for parameters $(I_\A^{(2)},J_\A^{(2)})_{\A\in \In^{(2)}},$ which satisfy 
\begin{align*}
|(I_\A^{(2)} \cup J_\A^{(2)})\setminus (I_\A^{(2)} \cap J_\A^{(2)})| = |(I_\A \cup J_\A)\setminus (I_\A \cap J_\A)|-2.
\end{align*}
Suppose $P:=\{\{u_1,v_1\},\dots,\{u_\eta,v_\eta\}\}$ is a partition of $I_1:=(I_\A \cup J_\A)\setminus (I_\A \cap J_\A)$ compatible with a 321-avoiding permutation of $\mathrm{S}_\eta,$ where $2\eta:= |I_1|.$ Corresponding to each of these partitions $P$, we define compositions of Chevalley operations via
\[
\So_{P}:=\So_{(u_\eta,v_\eta)}^{\mathcal{P}_\eta} \circ \dots \circ \So_{(u_1,v_1)}^{\mathcal{P}_1}
\]
where $\mathcal{P}_k:=(I_\A^{(k)},J_\A^{(k)})_{\A\in\In^{(k)}}$ is the parameter for the inequality obtained by operating $\So_{(u_{k-1},v_{k-1})}^{\mathcal{P}_{k-1}} \circ \dots \circ \So_{(u_1,v_1)}^{\mathcal{P}_1}$ on the inequality with parameter $\mathcal{P}_1=(I_\A,J_\A)_{\A\in\In},$ for $k=2,\dots,\eta.$ Theorem~\ref{Main-thm} now follows via Theorems~\ref{L:1} and \ref{L:2}, since each $\sum_{\A\in \In(P)}c_\A$ corresponds to the action of $\So_{P}$ over \eqref{quad_ineq}.
\end{proof}

The second application of Chevalley operations yields a novel result. Here we consider a smaller class of \eqref{quad_ineq} with $m=n$ such that the Pl\"ucker coordinates $I_\A, J_\A\in \In_{n}$ (see Definition~\ref{defn:main:3}). These inequalities essentially refer to quadratic matrix determinantal inequalities that involve only the principal minors. Recall that in Theorem~\ref{Main-thm} we classify all quadratic inequalities, and we do this via a set of nonnegativity conditions, each corresponding to a 321-avoiding permutation. Therefore, if we consider classifying a smaller class of quadratic inequalities (in which $I_\A, J_\A\in \In_{n}$) then expecting a smaller collection of nonnegativity conditions is natural. We show that this subclass of all nonnegativity conditions refer to partitions $P$ that are compatible with \textit{321-avoiding involutions} (compared to all 321-avoiding permutations required in Theorem~\ref{Main-thm}).

\begin{defn}[321-avoiding involutions]
Suppose $\eta\geq 1.$ A 321-avoiding permutation $\omega\in\mathrm{S}_{\eta}$ is called a 321-avoiding involution if it satisfies $\omega^{-1}=\omega.$
\end{defn}

It is well known that the 321-avoiding involutions are identified with Dyck words of semilength $\eta$ that are symmetric about the center of the word. The number of these is given by $\binom{\eta}{\floor{\eta/2}}.$


\begin{utheorem}[A classification via sums over 321-avoiding involutions]\label{TLnew}
Let $n,\eta\geq 1$ be integers and suppose $\mathcal{I}$ is an ordered finite indexing set. Suppose $I_{\alpha},J_{\alpha}\in \In_n$  (Definition~\ref{defn:main:3}) such that the multisets $I_{\alpha}\Cup J_{\alpha}$ are equal, for all $\A\in \In.$ Suppose $I:=(I_\A \cup J_\A) \setminus(I_\A \cap J_\A)$ and $2\eta:=|I|.$ Then for $c_{\alpha}\in \R,$
\begin{align}\label{TLnew:Eqn:1}
\sum_{\alpha\in \mathcal{I}} c_{\alpha} \Delta_{I_{\alpha}} \Delta_{J_{\alpha}} \geq 0 \quad\mbox{over}\quad \Gr^{\geq 0}(n,2n),
\end{align}
if and only if 
\begin{align}\label{TLnew:Eqn:2}
\sum_{\A\in \In(P)} c_{\A} \geq 0\quad \mbox{ for all partitions $P$ of $I$ compatible with a 321-avoiding involution in $\mathrm{S}_\eta$}.
\end{align}
\end{utheorem}

\begin{remark}[The novel refinement in Theorem~\ref{TLnew}]
We believe Theorem~\ref{TLnew} is a novel result. Furthermore, it raises an interesting question about a certain subspace of the span of Temperley--Lieb immanants (on which we elaborate in Section~\ref{futurework}).
\end{remark}

\begin{proof}[Proof of Theorem~\ref{TLnew}] Following Definition~\ref{compositive-Chev-Op-1}, define $\overline{\So}_{(u,v)}^{\mathcal{P}}:=~\So_{(v^*,u^*)}^{\mathcal{P}} \circ \So_{(u,v)}^{\mathcal{P}}$ for $\mathcal{P}:=(I_\A,J_\A)_{\A\in \In},$ whenever $u,v$ and $u^*,v^*$ are simultaneously consecutive in $(I_\A \cup J_\A) \setminus (I_\A \cap J_\A).$ One can see that for a consecutive pair $u,v\in I:=\{\ell_1<\dots<\ell_{2\eta}\},$ $\overline{\So}_{(u,v)}^{\mathcal{P}}$ applied on \eqref{TLnew:Eqn:1} yields an inequality of the form \eqref{TLnew:Eqn:1}, say for parameters $(I_\A^{(2)},J_\A^{(2)})_{\A\in \In^{(2)}}$ such that 
\begin{align*}
|(I_\A^{(2)} \cup J_\A^{(2)})\setminus (I_\A^{(2)} \cap J_\A^{(2)})| = 
\begin{cases}
|(I_\A \cup J_\A)\setminus (I_\A \cap J_\A)|-4 & \mbox{ if } \{u,v\}\neq \{\ell_\eta,\ell_{\eta+1}\}, \mbox{ and}\\
|(I_\A \cup J_\A)\setminus (I_\A \cap J_\A)|-2 & \mbox{ otherwise}.
\end{cases}
\end{align*}
Suppose $P:=\{\{u_1,v_1\},\dots,\{u_{\eta},v_{\eta}\}\}$ is a partition of $I$ compatible with a 321-avoiding involution in $\mathrm{S}_{\eta}.$ Corresponding to each such partition, we define compositions of Chevalley operations via
\[
\So_{P}:=\So_{(u_\eta,v_\eta)}^{\mathcal{P}_\eta} \circ \dots \circ \So_{(u_1,v_1)}^{\mathcal{P}_1}
\]
where $\mathcal{P}_k:=(I_\A^{(k)},J_\A^{(k)})_{\A\in\In^{(k)}}$ is the parameter for the inequality obtained by operating $\So_{(u_{k-1},v_{k-1})}^{\mathcal{P}_{k-1}} \circ \dots \circ \So_{(u_1,v_1)}^{\mathcal{P}_1}$ on $\mathcal{P}_1:=(I_\A,J_\A)_{\A\in\In},$ for $k=2,\dots,\eta.$ Since $P$ corresponds to a 321-avoiding involution, $\{u_j,v_j\} \in P$ if and only if $\{u_j^*,v_j^*\}\in P,$ for all $j\in [\eta].$ Therefore, the resultant of the composite operation immediately above can also be obtained as a resultant of the composite operation given by
\[
\So_{Q}:=\So_{({u_{\mu}'}^*,{v_{\mu}'}^*)}^{\mathcal{P}_\eta}\circ \So_{(u_{\mu}',v_{\mu}')}^{\mathcal{P}_\eta} \circ \dots \circ \So_{({u_1'}^*,{v_1'}^*)}^{\mathcal{P}_1}\circ \So_{(u_1',v_1')}^{\mathcal{P}_1} = \overline{\So}_{(u_{\mu}',v_{\mu}')}^{\mathcal{P}_\eta} \circ \dots \circ  \overline{\So}_{(u_1',v_1')}^{\mathcal{P}_1},
\]
where the elements in $Q:=
\big{\{}
\{u_1',v_1'\},\{{u_1'}^*,{v_1'}^*\},
\dots, 
\{u_{\mu}',v_{\mu}'\},\{ {u_{\mu}'}^*,{v_{\mu}'}^*\}
\big{\}}$ are exactly those of $P,$ but in a different order, and it may in fact be possible that $Q$ is a multiset. Nevertheless, Theorem~\ref{TLnew} follows via Theorems~\ref{L:3} and \ref{L:2}, since each $\sum_{\A\in \In(P)}c_\A$ refers to acting $\So_{Q}$ over \eqref{TLnew:Eqn:1}.
\end{proof}

\begin{remark}[Chevalley operations and Lam's Grassmann analogue of Temperley--Lieb immanants]\label{lam-TL-imm}
Lam \cite{lam2015totally,lam2014amplituhedron} discussed the Grassmann analogue of Temperley--Lieb immanants by considering \textit{partial noncrossing pairings}; see for instance \cite[Subsection~4.2]{lam2015totally} for details. The sums over 321-avoiding permutations and involutions (in Theorems~\ref{Main-thm} and \ref{TLnew}) can also be seen as sums over these partial noncrossing pairings and its subclass of ``symmetric'' partial noncrossing pairings, respectively. 
\end{remark}

\subsection{Log-supermodularity of Pl\"{u}cker coordinates}\label{Lam-sup-mod}

It is well known that every point in $\TGr$ can be represented by a network, and that $\TGr$ decomposes into \textit{positroid cells}. These positroid cells can be indexed by any of \textit{Grassmann necklaces}, \textit{bounded affine permutations}, and \textit{positroids}. (See \cite{postnikov2006total, lam2015totally} for details.) We focus on positroids here. A positroid is defined 
\[
\mbox{for each $V\in \TGr$ via } \mathcal{M}_V:=\big{\{}I\subset [m+n] \mbox{ with $m$-elements such that } \Delta_{I}(V)\neq 0 \big{\}}.
\]
For $I=\{i_1<\dots<i_m\},J:=\{j_1<\dots<j_m\}\subset [m+n],$ define binary operations
$$
\max(I,J):=\{\max(i_k,j_k):k\in [m]\} \quad\mbox{and}\quad \min(I,J):=\{\min(i_k,j_k):k\in [m]\}.
$$
The main results in this subsection are due to Lam \cite{lam2015totally}. The proofs by Lam were of a combinatorial spirit via his Grassmann analog of Temperley--Lieb immanants, while the ones we present are different and using the structure of Chevalley operations akin to cluster mutations.


\begin{utheorem}\label{logsupmod}
Each $\mathcal{M}_V$ forms a distributive lattice for the binary operations $(\max,\min)$.
\end{utheorem}
\begin{proof}
It is sufficient to prove the following inequality for all $m$-element subsets $I,J\subset [m+n]$:
\begin{align}\label{logsupmod-ineq-0}
\Delta_{I}\Delta_{J}\leq \Delta_{\min(I,J)} \Delta_{\max(I,J)} \quad \mbox{over}\quad\TGr.    
\end{align}
\noindent\textbf{Claim 1:} The unions with multiplicities $I\Cup J = \max(I,J)\Cup \min(I,J).$

\noindent Every element with multiplicity 1 in $I \Cup J$ obviously appears in exactly one of $\max(I,J)$ and $\min(I,J).$ Suppose $p\in I\cap J$ and call the $p\in I$ as $p_I$ and the $p\in J$ as $p_J.$ It can be seen that if either $p_I=\max(i_k,j_k),p_J=\max(i_l,j_l)$ or $p_I=\min(i_k,j_k),p_J=\min(i_l,j_l)$ then $k=l$. This completes this proof.\medskip

Suppose $2\eta=|(I\cup J)\setminus (I\cap J)|,$ and let $\Phi:(I\cup J)\setminus (I\cap J) \to [2\eta]$ be the unique order preserving map. Consider
\begin{align}\label{logsupmod-ineq-1}
\Delta_{\Phi(I\setminus(I\cap J))}\Delta_{\Phi(J\setminus(I\cap J))}\leq \Delta_{\Phi(\min(I,J)\setminus(I\cap J))} \Delta_{\Phi(\max(I,J)\setminus(I\cap J))} \quad \mbox{over}\quad\Gr^{\geq 0}(\eta,2\eta).    
\end{align}
In light of Theorem~\ref{L:2}, \eqref{logsupmod-ineq-0} and \eqref{logsupmod-ineq-1} are simultaneously true.\medskip

\noindent\textbf{Claim 2:} If $K:=\Phi(I\setminus(I\cap J))$ and $L:=\Phi(J\setminus(I\cap J))$ then $\min(K,L)=\Phi(\min(I,J)\setminus(I\cap J))$ and $\max(K,L)=\Phi(\max(I,J)\setminus(I\cap J)).$

\noindent This follows using a similar argument as in Claim 1, since $\Phi$ is order preserving.\medskip

Therefore, henceforth we assume that $m=n$ and $I\cap J=\emptyset$ in \eqref{logsupmod-ineq-1}. We need to now show that for arbitrary consecutive integers $u,v\in [2m],$ the action of the Chevalley operation $\So_{(u,v)}$ on \eqref{logsupmod-ineq-1} yields a valid inequality.\medskip

\noindent\textbf{Case 1:} $u,v \in I$ and $u\in \max(I,J),v\in \min(I,J),$ or $u\in \min(I,J),v\in \max(I,J).$

\noindent The expression obtained after applying $\So_{(u,v)}$ over \eqref{logsupmod-ineq-1} is of the form $0 \leq \Delta_{\min(I,J)(u,v)} \Delta_{\max(I,J)(u,v)},$ which is valid. \medskip

\noindent\textbf{Case 2:} $u\in I,v\in J$ and $u\in \max(I,J),v\in \min(I,J),$ or $u\in \min(I,J),v\in \max(I,J).$

\noindent The expression obtained after applying $\So_{(u,v)}$ on \eqref{logsupmod-ineq-1} is of the form
\begin{align}\label{logsupmod-ineq-2}
\Delta_{I(u,v)}\Delta_{J(u,v)}\leq \Delta_{\min(I,J)(u,v)} \Delta_{\max(I,J)(u,v)} = \Delta_{\min(I(u,v),J(u,v))} \Delta_{\max(I(u,v),J(u,v))}.
\end{align}
Using an argument similar to Claim 2, one can see that \eqref{logsupmod-ineq-2} over $\Gr(m,2m)$ is equivalent to
\begin{align}\label{logsupmod-ineq-3}
\Delta_{I'}\Delta_{J'}\leq \Delta_{\min(I',J')} \Delta_{\max(I',J')}  \quad \mbox{over}\quad \Gr(m-1,2m-2)
\end{align}
where $I'= (i_1',\dots,i_{m-1}'),~J'=(j_1',\dots,j_{m-1}')$ with
\[
i_{k}'=\begin{cases} i_k & \mbox{ if }i_k<u, \\ i_{k}-2 & \mbox{ if } i_k> u,\end{cases} \quad\mbox{and}\quad j_{k}'=\begin{cases} j_k & \mbox{ if }i_k<v, \\ j_{k}-2 & \mbox{ if } j_k> v.\end{cases}
\]\medskip

\noindent\textbf{Case 3:} $u\in I,v\in J$ and $u,v\in \max(I,J),$ or $u,v\in \min(I,J).$

\noindent Since $u,v$ are consecutive, this case is not possible.\medskip

To summarize the arguments above: we proved that the action of any $\So_{(u,v)}$ over \ref{logsupmod-ineq-1} either yields an exprresion which is trivially valid (Case 1), or an expression of the form \eqref{logsupmod-ineq-1} itself in the smaller TNN Grassmannian (Case 2). In addition, it is easy to see that \eqref{logsupmod-ineq-1} trivially holds for $m=n=1$ as $I=\{1\},~J=\{2\}$ and $\min(I,J)=\{1\},~\max(I,J)=\{2\}.$ Therefore, by induction, inequality~\eqref{logsupmod-ineq-3} holds over $\Gr^{\geq 0}(m-1,2m-2),$ and so the expression~\eqref{logsupmod-ineq-1} is valid. This completes the proof. 
\end{proof}

Now we discuss the log-supermodularity. Define a positroid cell 
\[
\mbox{for each $V\in \TGr$ via
 }\Pi_{\mathcal{M}_V,>0}:=\big{\{} W\in \TGr: \mathcal{M}_W = \mathcal{M}_V  \big{\}}.
\]
It is well known that 
\begin{align*}
\mbox{for all }W,V\in \TGr \quad\mbox{either}&\quad \Pi_{\mathcal{M}_W,>0} = \Pi_{\mathcal{M}_V,>0}\quad \mbox{or}\quad \Pi_{\mathcal{M}_W,>0} \cap \Pi_{\mathcal{M}_V,>0}=\emptyset,\\
\mbox{and}&\quad \TGr=\cup_{V\in \TGr} \Pi_{\mathcal{M}_V,>0}.    
\end{align*}
Moreover, each positroid cell can be identified by certain minimal \textit{plabic graphs}, which in turn are identifiable via bounded affine permutations induced by certain \textit{zig-zag paths} on the minimal plabic graphs. We refer the reader to \cite{postnikov2006total, lam2015totally} for the details. Here we focus on the positroids of the positroid cells. 

Theorem~\ref{logsupmod} shows that each positroid $(\mathcal{M}_V,\max,\min)$ forms a distributive lattice. This has an immediate corollary regarding the log-supermodularity of Pl\"ucker coordinates over each positroid. Suppose $(\mathcal{L},\vee,\wedge)$ is a distributive lattice. We call a function $f:\mathcal{L}\to \R$ \textit{supermodular} provided $f(x\vee y)+ f(x\wedge v)\geq f(x)+ f(y)$ for all $x,y\in \mathcal{L}.$ A function $g:\mathcal{L}\to \R_{>0}$ is called \textit{log-supermodular} if $\log g$ is supermodular. An immediate consequence of Theorem~\ref{logsupmod} is the following.

\begin{ucor}\label{cor:I}
For $W \in \Pi_{\mathcal{M}_V,>0}$ with each $\Delta_I\geq 0,$ the map $I\mapsto \Delta_{I}(W)$ is log-supermodular on $\mathcal{M}_V.$
\end{ucor}

Corollary~\ref{cor:I} has several other applications. For one, it implies the ``numerical'' part of the main 2007 result of Lam--Postnikov--Pylyavskyy. More precisely, in \cite[Theorem~11]{LPP}, the authors show the following numerical result (which they also upgrade to Schur positivity for generalized Jacobi--Trudi matrices).

\begin{theorem}[Lam--Postnikov--Pylyavskyy \cite{LPP}]\label{LPP:thm}
Let $n,k\geq 1$ be integers, and suppose $P,Q,R,S$ be $k$-element subsets of $[n].$ Then we have 
\[
\det A_{P,Q} \det A_{R,S} \leq \det A_{\max(P,R),\max(Q,S)} \det A_{\min(P,R),\min(Q,S)}
\]
for all $n\times n$ generalized Jacobi--Trudi matrices $A$ (more generally, for all TNN matrices $A$).
\end{theorem}
\begin{proof}
Use \eqref{Plucker-embed} to transform the inequality in Theorem~\ref{LPP:thm} into \eqref{logsupmod-ineq-0}. Now apply Corollary~\ref{cor:I}.
\end{proof}

Theorem~\ref{LPP:thm} is a powerful result, with many consequences, e.g. Okounkov's log-concavity conjecture \cite{Oko} for skew Schur polynomials -- numerically, i.e.\ when evaluated in finitely many variables on the positive orthant.

We end this part by mentioning another application of the log-supermodularity in Corollary~\ref{cor:I}.

\begin{theorem}[Khare--Tao \cite{KT}]
See \cite{KT} for notations; let $n\geq 1$ be an integer and suppose $\lambda \supseteq \mu$ are two partitions. Then the Schur polynomial ratio
\[\frac{s_{\lambda}(u_1,\dots,u_n)}{s_{\mu}(u_1,\dots,u_n)}, \qquad u_1,\dots,u_n\in (0,\infty)
\]
is non-decreasing in each variable. More generally, let $\lambda \geq \mu \geq {\bf 0}$ be real tuples, then the map 
$$
{\bf u} \mapsto \frac{\det (u_i^{\lambda_j})_{i,j=1}^n}{\det (u_i^{\mu_j})_{i,j=1}^n}, \qquad {\bf u}\in (0,\infty)^n_{\neq}
$$
is non-decreasing in each variable, where $S^n_{\neq}$ denotes ordered $n$-tuples with pairwise distinct coordinates in $S$.
\end{theorem}
This theorem was used to (perhaps surprisingly) obtain sharp quantitative estimates for classes of polynomials preserving positive semidefinite matrices of a fixed size, when applied entrywise. The result follows (by the quotient rule for differentiation$!$) from:
\begin{prop}[Khare--Tao \cite{KT}]\label{prop:KT}
Let $n\geq 1$ be an integer and suppose $\lambda \supseteq \mu$ are two partitions. Fix scalars $u_1,\dots,u_{n-1}\in (0,\infty),$ and define 
    \[
f_\lambda(x):=s_{\lambda}(u_1,\dots,u_{n-1},x)    
    \]
    and similarly define $f_{\mu}(x).$ Then the polynomial $f_{\mu}f_{\lambda}' - f_{\mu}'f_{\lambda}$ has nonnegative real coefficients.
\end{prop}

As explained by Khare--Tao~\cite{KT}, Proposition~\ref{prop:KT} is (modulo some bookkeeping) a straightforward consequence of Corollary~\ref{cor:I} / Lam~\cite{lam2015totally}.

\subsection{Ordering the induced character immanants}\label{BJsection}

A weakly decreasing positive integer sequence $\lambda:=(\lambda_1,\dots,\lambda_\ell)$ is called a partition of $n$ if $\lambda_1+\dots+\lambda_\ell = n;$ denote this by $\lambda \vdash n.$ Recall the definition of an $f$-immanant in \eqref{eq:immdef}. Given a partition $\lambda\vdash n,$ suppose $\chi^{\lambda} \equiv \sgn \uparrow_{\mathrm{S}_{\lambda}}^{\mathrm{S}_{n}}$ is the \textit{induced sign character}, where $\mathrm{S}_{\lambda}$ is the Young subgroup of $\sn$ indexed by $\lambda.$ The induced character $\chi^{\lambda}$ yields the \textit{induced character immanant}, the simple formula for which is given by Littlewood--Merris--Watkins \cite{LittlewoodTGC,MerrisWatkins}:
\begin{align*}
\imm_{\chi^{\lambda}}(\mathsf{{\bf x}})=\sum_{(I_1,\dots,I_\ell)} \det {\mathsf{{\bf x}}}_{I_1,I_1}\dots \det {\mathsf{{\bf x}}}_{I_\ell,I_\ell} \quad \mbox{for}\quad\mathsf{{\bf x}}=(x_{ij})_{n\times n}
\end{align*}
where $\lambda=(\lambda_1,\dots,\lambda_\ell) \vdash n$ and the sum is over all pairwise disjoint subsets $I_j \subset [n]$ with $|I_j|=\lambda_j.$ Define a partial order $\preceq$ on partitions of $n$ via the \textit{majorization order} (padding partitions by trailing zeros to equalize their lengths) $\lambda = (\lambda_1,\dots,\lambda_r) \preceq \mu = (\mu_1,\dots,\mu_s)$ provided $\sum_{k=1}^{j}\lambda_k \leq \sum_{k=1}^{j}\mu_k$ for all $j.$ Due to the works of Barrett--Johnson \cite{barrett1993majorization} and Skandera--Soskin \cite{skan2022sosk},  we know that majorization yields a partial order on induced character immanants,
\begin{align}\label{bj-eqn2}
\lambda \preceq \mu  \quad \implies \quad {\binom{n}{\lambda_1,\dots,\lambda_r}}^{-1} \imm_{\chi^{\lambda}} (A) ~ \geq ~ {\binom{n}{\mu_1,\dots,\mu_s}}^{-1} \imm_{\chi^{\mu}} (A)
\end{align}
for all $A_{n\times n}$ real positive semidefinite and totally nonnegative matrices. 
It is known that if $\lambda \preceq \mu$ then there exists a sequence of partitions $\lambda=:\delta_0 \preceq \delta_1 \preceq \dots \preceq \delta_k:=\mu$ such that each pair $\delta_j,\delta_{j+1}$ differs only at two coordinates. Therefore, via telescoping, the proof of \eqref{bj-eqn2} reduces to the special case $\lambda=(k,n-k)$ and $\mu=(k+1,n-k-1)$ for all $k=0,1,\dots, \floor{n/2}-1.$ 

We now provide a new proof of this special case (and hence of of \eqref{bj-eqn2}) for TNN matrices, using Chevalley operations (which are akin to cluster mutations, in contrast to the combinatorial proof in \cite{skan2022sosk}). 
We begin by reformulating it using \eqref{Plucker-embed} and Theorem~\ref{Ineq-equivalence}:


\begin{utheorem}\label{T:BJ}
Given integers $n\geq 2$ and $0\leq k\leq n,$ define 
\begin{align}\label{defn:BJ:1}
\In_{n,k}:=\big{\{}I:=P\cup \big{(} [n+1,2n]\setminus (2n+1-P) \big{)}: P\subseteq [n] \mbox{ and }|P|=k\big{\}}.
\end{align}
Then for any $n\geq 2$ and $0\leq k\leq \floor{n/2}-1,$ we have
\begin{align}\label{T:BJ:Eqn:1}
{\binom{n}{k}}^{-1}\sum_{I\in \In_{n,k}} \Delta_{I} \Delta_{I^{\mathsf{c}}} \leq {\binom{n}{k+1}}^{-1}\sum_{I\in \In_{n,k+1}} \Delta_{I} \Delta_{I^{\mathsf{c}}} \quad\mbox{over}\quad \Gr^{\geq 0}(n,2n). 
\end{align}
\end{utheorem}

\begin{example}\label{example:new:added}
Here we discuss \eqref{T:BJ:Eqn:1} for small values of $n.$ Suppose $n=2;$ then $N(2):=\floor{n/2}-1=0,$ and therefore $k=0$ and $k+1=1.$ The corresponding \eqref{T:BJ:Eqn:1} in the matrix minor form is:
\begin{align*}
    \binom{2}{0}^{-1}\det A \leq \binom{2}{1}^{-1}(a_{11}a_{22}+a_{22}a_{11}) \iff \det A \leq a_{11}a_{22}
\end{align*}
for all $2\times 2$ $A:=(a_{ij})$ TNN. This is easy to verify.\medskip

Now suppose $n=3.$ Then $N(3):=\floor{3/2}-1=0$; so $k=0$ and $k+1=1.$ Therefore the corresponding \eqref{T:BJ:Eqn:1} in the matrix determinant form is the following, where each $A_{K}$ denotes the principal submatrix indexed by $K\times K$:
\begin{align*}
    \det A &\leq \frac{1}{3}\big{(} a_{11}\det A_{23}+a_{22}\det A_{13}+ a_{33} \det A_{12}\big{)}
\end{align*}
for all $3\times 3$ TNN $A:=(a_{ij})$. To show that this forms an inequality we use Theorems~\ref{L:3} and \ref{L:2} over its ``Pl\"ucker coordinated'' form (see \eqref{Plucker-embed} and Theorem~\ref{Ineq-equivalence}) given by:
\begin{align*}
    \Delta_{123}\Delta_{456}\leq \frac{1}{3}\big{(} \Delta_{145}\Delta_{236} + \Delta_{246}\Delta_{135}+\Delta_{356}\Delta_{124}. \big{)}
\end{align*}
To show that the above forms an inequality over $\Gr^{\geq}(3,6)$ we apply Theorem~\ref{L:3} on it. This entails applying composite Chevalley operations $\So_{(v^*,u^*)}\circ\So_{(u,v)}$ on it for consecutive $u,v\in \{1,2,3,4\}.$ If the consecutive $u,v\in \{1,2,3\}$ then the expression on the lower side of the inequality vanishes, and we are left with the inequality involving products of pairs of Pl\"ucker coordinates on the higher side, which is always true over a TNN Grassmannian. So we consider the case when $u=3$ and $v=4$ (without loss of generality). Applying $\So_{(3,4)}\circ\So_{(3,4)}$ on the immediately above expression gives:
\begin{align*}
    \Delta_{124}\Delta_{456}\leq \frac{1}{3}\big{(} \Delta_{145}\Delta_{246} + \Delta_{246} \Delta_{145} + \Delta_{456}\Delta_{124} \big{)}.
\end{align*}
Using Theorem~\ref{L:2} we have that the above is equivalent to the following over $\Gr^{\geq 0}(2,4)$:
\begin{align*}
    \Delta_{12}\Delta_{34} \leq \frac{1}{3}\big{(} \Delta_{13}\Delta_{24} + \Delta_{24}\Delta_{13}+\Delta_{34}\Delta_{12} \big{)}.
\end{align*}
The above expression in the determinant form looks like the following for all $2\times 2$ TNN $A=(a_{ij})$:
\begin{align*}
    \det A \leq \frac{1}{3}\big{(} 2a_{11}a_{22}+\det A \big{)} \iff \det A \leq a_{11}a_{22},
\end{align*}
and so it holds, and therefore the expression that we started with over $\Gr^{\geq 0}(3,6)$ holds. This completes the $n=3$ case.\medskip

Next, let $n=4.$ Then $N(4):=\floor{4/2}-1=1$; so $k=0,1.$ We first consider the $k=0$ case. The expression~\eqref{T:BJ:Eqn:1} in the matrix determinant form looks like:
\begin{align*}
    \det A \leq \frac{1}{4}\big{(} a_{11} \det A_{234} + a_{22}\det A_{134} + a_{33}\det A_{124} + a_{44}\det A_{123} \big{)}
\end{align*}
for all $4\times 4$ TNN $A:=(a_{ij}).$ The Pl\"ucker coordinated version of the above expression is over $\Gr^{\geq 0}(4,8)$ and is given by:
\begin{align*}
    \Delta_{1234}\Delta_{5678}\leq \frac{1}{4} \big{(} \Delta_{1567} \Delta_{2348} + \Delta_{2568}\Delta_{1347} + \Delta_{3578} \Delta_{1246} + \Delta_{4678}\Delta_{1235} \big{)}.
\end{align*}
To show that the above is valid over $\Gr^{\geq 0}(4,8)$ we use Theorem~\ref{L:3}; it is easy to see that applying $\So_{(v^*,u^*)}\circ \So_{(u,v)}$ over the above for any consecutive $u,v\in \{1,2,3,4\}$ yields an expression with the lower side is zero; which implies that the inequality is trivially true. So we assume that $u=4$ and $v=5$ (without loss of generality). Applying $\So_{(4,5)}\circ \So_{(4,5)}$ on the expression above gives:
\begin{align*}
    \Delta_{1235}\Delta_{5678} \leq \frac{1}{4} \big{(} \Delta_{1567}\Delta_{2358} + \Delta_{2568}\Delta_{1357} + \Delta_{3578}\Delta_{1256} + \Delta_{5678}\Delta_{1235} \big{)}.
\end{align*}
Via Theorem~\ref{L:2} we write an equivalent expression of the above over $\Gr^{\geq 0}(3,6)$ given by:
\begin{align*}
    \Delta_{123}\Delta_{456} &\leq \frac{1}{4}\big{(} \Delta_{145}\Delta_{236} + \Delta_{246} \Delta_{135} + \Delta_{356}\Delta_{124} + \Delta_{456}\Delta_{123}\big{)} \\
    \iff \Delta_{123}\Delta_{456} &\leq \frac{1}{3}\big{(} \Delta_{145}\Delta_{236} + \Delta_{246} \Delta_{135} + \Delta_{356}\Delta_{124}\big{)}.
\end{align*}
This is an inequality over $\Gr^{\geq 0}(3,6)$ which we have already proved above. This completes the $k=0$ case for $n=4.$ We next proceed for the $k=1$ case. The expression (that we wish to prove) for all $4\times 4$ TNN $A=(a_{ij})$ is given by
\begin{align*}
    &\frac{1}{4}\big{(} a_{11}\det A_{234} + a_{22}\det A_{134} + a_{33}\det A_{124} + a_{44}\det A_{123}\big{)}\\ &\quad \leq \frac{1}{3}\big{(} \det A_{12} \det A_{34} + \det A_{13} \det A_{24} + \det A_{14} \det A_{23}\big{)}.
\end{align*}
This converted into Pl\"ucker coordinated form looks like the following over $\Gr^{\geq 0}(4,8)$:
\begin{align*}
    &\frac{1}{4}\big{(} \Delta_{1567} \Delta_{2348} + \Delta_{2568}\Delta_{1347} + \Delta_{3578} \Delta_{1246} + \Delta_{4678} \Delta_{1235} \big{)} \\
    &\quad\leq \frac{1}{3} \big{(} \Delta_{1256} \Delta_{3478} + \Delta_{1357}\Delta_{2468} + \Delta_{1467}\Delta_{2358} \big{)}
\end{align*}
As we did before, apply $\So_{(7,8)}\circ \So_{(1,2)}$ on above, to obtain:
\begin{align*}
    \frac{1}{4} \big{(} \Delta_{2568} \Delta_{2348} + \Delta_{2568} \Delta_{2348} \big{)} \leq \frac{1}{3}\big{(} \Delta_{2358}\Delta_{2468} + \Delta_{2468}\Delta_{2358} \big{)}
\end{align*}
Employing Theorem~\ref{L:2} on above, we get the following expression over $\Gr^{\geq 0}(2,4)$:
\begin{align*}
    \frac{1}{2}\Delta_{12}\Delta_{34} \leq \frac{2}{3}\Delta_{13}\Delta_{24} \iff \det A \leq (4/3) a_{11}a_{22},    
\end{align*}
for all $2\times 2$ TNN $A,$ which we know is true. Similar is the case when applying $\So_{(6,7)}\circ \So_{(2,3)}$ and $\So_{(5,6)}\circ \So_{(3,4)},$ so skip the details, and proceed to apply $\So_{(4,5)}\circ \So_{(4,5)},$ and obtain the following expression over $\Gr^{\geq 0}(4,8)$:
\begin{align*}
    &\frac{1}{4}\big{(} \Delta_{1567} \Delta_{2358} + \Delta_{2568}\Delta_{1357} + \Delta_{3578} \Delta_{1256} + \Delta_{5678} \Delta_{1235} \big{)} \\
    &\quad\leq \frac{1}{3} \big{(} \Delta_{1256} \Delta_{3578} + \Delta_{1357}\Delta_{2568} + \Delta_{1567}\Delta_{2358} \big{)}
\end{align*}
Now employ Theorem~\ref{L:2} to write the equivalent expression in the smaller dimension:
\begin{align*}
    &\frac{1}{4}\big{(} \Delta_{145} \Delta_{236} + \Delta_{246}\Delta_{135} + \Delta_{356} \Delta_{124} + \Delta_{456} \Delta_{123} \big{)} \\
    &\quad\leq \frac{1}{3} \big{(} \Delta_{124} \Delta_{356} + \Delta_{135}\Delta_{246} + \Delta_{145}\Delta_{236} \big{)}
\end{align*}
over $\Gr^{\geq 0}(3,6).$ Again, this is an expression which we have proved already. This completes the $n=4$ case. We now explain why the general case holds.
\end{example}

\begin{proof}[Proof of Theorem~\ref{T:BJ}]
Recall Definition~\ref{defn:main:3}, and observe that $\In_{n,k},\In_{n,k+1}\subseteq \In_n.$ Theorem~\ref{L:3} implies that inequality~\eqref{T:BJ:Eqn:1} holds if and only if the action of composite operation $\So_{(v^*,u^*)}\circ \So_{(u,v)}$ on \eqref{T:BJ:Eqn:1}, for all consecutive $u,v\in [n+1],$ yields a valid inequality. In light of Example~\ref{example:new:added} we assume that \eqref{T:BJ:Eqn:1} holds for all smaller values of $n.$ Using arguments similar to the $n=2,3,4$ and $k=0$ in Example~\ref{example:new:added} one can verify \eqref{T:BJ:Eqn:1} for $k=0$ and larger $n$; we skip these details. We address the cases when $k\geq 1$ and $n\geq 5$. $\So_{(v^*,u^*)}\circ\So_{(u,v)}$ applied on \eqref{T:BJ:Eqn:1} yields
\begin{align}\label{T:BJ:Eqn:2}
\frac{1}{\binom{n}{k}}\sum_{I\in (\In_{n,k} (u,v))(v^*,u^*)}& \Delta_{(I(u,v))(v^*,u^*)} \Delta_{(I^{\mathsf{c}}(u,v))(v^*,u^*)}\\
&\hspace*{2cm}\leq \frac{1}{\binom{n}{k+1}}\sum_{I\in (\In_{n,k+1} (u,v))(v^*,u^*)} \Delta_{(I(u,v))(v^*,u^*)} \Delta_{(I^{\mathsf{c}}(u,v))(v^*,u^*)},\nonumber
\end{align}
where $(\In_{n,\ell}(u,v))(v^*,u^*),$ $\ell \in \{k,k+1\},$ refers to those $\ell$ element subsets $P$ (see~\eqref{defn:BJ:1}) such that 
\begin{itemize}
    \item either $u\in P$ and $v\not\in P$ or $v\in P$ and $u\not\in P$; when consecutive $u,v\in [n].$
    \item either $n\in P$ or $n\not\in P$; when $u=n$ and $v=n+1$ (without loss of generality).
\end{itemize}
\noindent\textbf{Case 1:} $u,v\in [n].$ We write \eqref{T:BJ:Eqn:1} for $n'=n-2$ and $k'=k-1$: 
\begin{align}\label{T:BJ:new1}
\frac{1}{\binom{n-2}{k-1}}\sum_{I\in\In_{n-2,k-1}} \Delta_{I} \Delta_{I^{\mathsf{c}}} \leq \frac{1}{\binom{n-2}{k}}\sum_{I\in\In_{n-2,k}} \Delta_{I} \Delta_{I^{\mathsf{c}}} \quad\mbox{over}\quad \Gr^{\geq 0}(n-2,2n-4).
\end{align}
Since $k\leq \frac{n}{2} - 1$ we have $
\frac{k+1}{k}\cdot\frac{n-k-1}{n-k} \geq \frac{n^2}{n^2-4} > 1.$ Using this, \eqref{T:BJ:new1} is appended, and we have
\begin{align*}
\frac{1}{\binom{n-2}{k-1}}\sum_{I\in\In_{n-2,k-1}} \Delta_{I} \Delta_{I^{\mathsf{c}}} \leq \frac{1}{\binom{n-2}{k}}\sum_{I\in\In_{n-2,k}} \Delta_{I} \Delta_{I^{\mathsf{c}}} \leq \frac{\frac{k+1}{k}\cdot\frac{n-k-1}{n-k}}{\binom{n-2}{k}}\sum_{I\in\In_{n-2,k}} \Delta_{I} \Delta_{I^{\mathsf{c}}}.
\end{align*}
This implies,
\begin{eqnarray}\label{T:BJ:Eqn:3}
&  \frac{1}{\binom{n-2}{k-1}} \sum_{I\in\In_{n-2,k-1}} \Delta_{I} \Delta_{I^{\mathsf{c}}} & \leq \frac{\frac{k+1}{k}\cdot \frac{n-k-1}{n-k}}{\binom{n-2}{k}}\sum_{I\in\In_{n-2,k}} \Delta_{I} \Delta_{I^{\mathsf{c}}} \nonumber \\ 
\iff &\frac{1}{\frac{n(n-1)}{k(n-k)}\binom{n-2}{k-1}}\sum_{I\in\In_{n-2,k-1}} \Delta_{I} \Delta_{I^{\mathsf{c}}} &\leq \frac{1}{\frac{n(n-1)}{(k+1)(n-k-1)}\binom{n-2}{k}}\sum_{I\in\In_{n-2,k}} \Delta_{I} \Delta_{I^{\mathsf{c}}} \nonumber \\
\iff & \frac{1}{\binom{n}{k}} \sum_{I\in\In_{n-2,k-1}} \Delta_{I} \Delta_{I^{\mathsf{c}}} & \leq \frac{1}{\binom{n}{k+1}} \sum_{I\in\In_{n-2,k}} \Delta_{I} \Delta_{I^{\mathsf{c}}}\quad\mbox{over}\quad \Gr^{\geq 0}(n-2,2n-4).
\end{eqnarray}
Using Theorems~\ref{L:2} and \ref{L:3} we can see that inequality \eqref{T:BJ:Eqn:2} is equivalent to inequality~\eqref{T:BJ:Eqn:3}. \medskip

\noindent\textbf{Case 2:} $u=n$ and $v =n+1.$ It is not very difficult to see that equation~\eqref{T:BJ:Eqn:2} can be written as the following over $\Gr^{\geq 0}(n-1,2n-2)$:
\begin{align}\label{T:BJ:Case2:new}
\binom{n}{k}^{-1} \bigg{(} \sum_{I\in \In_{n-1,k}} \Delta_I \Delta_{I^\mathsf{c}} + \sum_{I\in \In_{n-1,k-1}} \Delta_I \Delta_{I^\mathsf{c}} \bigg{)} \leq \binom{n}{k+1}^{-1} \bigg{(} \sum_{I\in \In_{n-1,k+1}} \Delta_I \Delta_{I^\mathsf{c}} + \sum_{I\in \In_{n-1,k}} \Delta_I \Delta_{I^\mathsf{c}} \bigg{)}.
\end{align}
We claim that the above inequality is true. We introduce notation for ease: define $$\mathcal{E}(n,k):=\sum_{I\in \In_{n,k}}\Delta_I \Delta_{I^{\mathsf{c}}}.$$

\noindent\textbf{Case 2(a):} Suppose $k$ is sufficiently small, i.e., $k\leq \floor{\frac{n-1}{2}}-1.$ Then by induction we have that 
\begin{align}\label{T:BJ:Case2:new:1}
    \binom{n-1}{k-1}^{-1}\mathcal{E}(n-1,k-1)&\leq \binom{n-1}{k+1}^{-1}\mathcal{E}(n-1,k+1)\nonumber\\ \iff     \mathcal{E}(n-1,k-1)&\leq \frac{k(k+1)}{(n-k)(n-k-1)}\mathcal{E}(n-1,k+1) 
\end{align}
over $\Gr^{\geq 0}(n-1,2n-2).$ Similarly we have
\begin{align}\label{T:BJ:Case2:new:2}
    \mathcal{E}(n-1,k)\leq \frac{k+1}{n-k-1}\mathcal{E}(n-1,k+1)
\end{align}
over $\Gr^{\geq 0}(n-1,2n-2).$ Equation~\eqref{T:BJ:Case2:new} can be re-written as
\begin{align*}
    \binom{n}{k+1}\binom{n}{k}^{-1}\mathcal{E}(n-1,k-1) + \bigg{(}\binom{n}{k+1}\binom{n}{k}^{-1}-1\bigg{)}\mathcal{E}(n-1,k) &\leq \mathcal{E}(n-1,k+1)\\
    \iff 
    \frac{n-k}{k+1}\mathcal{E}(n-1,k-1) + \frac{n-2k-1}{k+1}\mathcal{E}(n-1,k) &\leq \mathcal{E}(n-1,k+1)
\end{align*}
To verify this equivalent expression, use the two inequalities in \eqref{T:BJ:Case2:new:1} and \eqref{T:BJ:Case2:new:2} and substitute the left hand side. This completes the proof via induction.\medskip

\noindent\textbf{Case 2(b):} Suppose $k$ is not small enough, i.e., $n$ is even and $k=(n/2)-1.$ Then \eqref{T:BJ:Case2:new} simplifies into the following over $\Gr^{\geq 0}(n-1,2n-2)$:
\begin{align*}
    \binom{n}{k}^{-1}\big{(}\mathcal{E}(n-1,k) + \mathcal{E}(n-1,k-1) \big{)} &\leq 2\binom{n}{k+1}^{-1}\mathcal{E}(n-1,k)\\
    \iff \mathcal{E}(n-1,k-1)&\leq \bigg{(}2\binom{n}{k}\binom{n}{k+1}^{-1}-1\bigg{)}\mathcal{E}(n-1,k).
\end{align*}
Since $2k-n+2=0,$ a direct computation shows that $2\binom{n}{k}\binom{n}{k+1}^{-1}-1 = \binom{n-1}{k-1}\binom{n-1}{k}^{-1},$ which completes the proof via induction.
\end{proof}

We conclude this section with a note on the \textit{closed-ness} of the collection of inequalities in \eqref{T:BJ:Eqn:1} -- also known as the Barrett--Johnson inequalities \cite{skan2022sosk} -- under the composite Chevalley operations $\So_{(v^*,u^*)}\circ \So_{(u,v)}.$ To be more precise: one may look at Example~\ref{example:new:added} and the proof of Theorem~\ref{T:BJ} and see that upon employing Theorems~\ref{L:3} and \ref{L:2} on a Barrett--Johnson inequality \eqref{T:BJ:Eqn:1}, the resultant expression is validated via a Barrett--Johnson inequality itself, in a smaller dimension, and without having to look outside this very specific collection of inequalities.

\section{Proof of main results}\label{WThm_TNNGrass}

Henceforth, we use $\I$ to denote the identity matrix, and $\El_{u,v}$ to refer to the elementary matrix with $1$ at $(u,v)$ position, and zero otherwise. The sizes of the matrices will be obvious from the context or otherwise stated. We respectively use $\R_{\geq 0}$ and $\R_{+}$ to denote the nonnegative and positive real numbers. We begin with the following well known theorem in total positivity. 

\begin{theorem}[Loewner \cite{Loewner55}, Whitney \cite{W}, Berenstein--Fomin--Zelevinsky \cite{BFZ99}]\label{TN-classification}
Let $m\geq 2,$ and consider 
$$
\mathrm{\bf w}:=(w_{j,k})_{1\leq j\leq m-1,j\leq k\leq m-1}\in \R_+^{\binom{m}{2}},\quad\mathrm{\bf w'}:=(w_{j,k+1}')_{1\leq j\leq m-1,j\leq k\leq m-1}\in \R_+^{\binom{m}{2}},\quad \mathrm{\bf d}:=(d_j)_{1\leq j\leq m}\in \R_+^m.
$$
Define $m\times m$ matrices $A$ for all $(\mathrm{\bf w},\mathrm{\bf w'},\mathrm{\bf d})\in \R_+^{m^2}$ via
\begin{align}\label{TN-classification-eqn}
A:=A(\mathrm{\bf w},\mathrm{\bf w'},\mathrm{\bf d}):=\prod_{j=1}^{m-1} \prod_{k=m-1}^{j} \big{(} \I+w_{j,k}\El_{k+1,k} \big{)}~\prod_{j=m-1}^{1} \prod_{k=j}^{m-1} \big{(} \I+ w'_{j,k+1}\El_{k,k+1} \big{)}~D
\end{align}
where $D=\diag(d_1,\dots,d_m).$ 
The map \eqref{TN-classification-eqn} defines a diffeomorphism
\begin{align}\label{diffeomorphism}
\Psi:\R_{+}^{{m^2}}\to \{m\times m \mbox{ totally positive matrices}\}\quad\mbox{via}\quad (\mathrm{\bf w},\mathrm{\bf w'},\mathrm{\bf d}) \xmapsto{\Psi} A(\mathrm{\bf w},\mathrm{\bf w'},\mathrm{\bf d}).
\end{align}
Moreover, the factors in \eqref{TN-classification-eqn} can be shuffled in any arbitrary manner for the corresponding diffeomorphism $\Psi$ in \eqref{diffeomorphism} to exist. Furthermore, a real matrix $A_{m\times m}$ is nonsingular totally nonnegative if and only if it has the (shuffled) form of \eqref{TN-classification-eqn} for $(\mathrm{\bf w},\mathrm{\bf w'},\mathrm{\bf d})\in \R_{\geq 0}^{\binom{m}{2}} \times \R_{\geq 0}^{\binom{m}{2}} \times \R_{+}^{m}.$ Finally, the collection of $m\times m$ nonsingular totally nonnegative matrices is dense in the set of $m\times m$ totally nonnegative matrices, and so is the collection of all $m\times m$ totally positive matrices.
\end{theorem}

We now prove the Grassmann analogue of Theorem~\ref{TN-classification}. One may refer to Lam \cite{lam2015totally} for the combinatorial treatment of this analogue. Here we prove it according to our requirements using Theorem~\ref{TN-classification}. We begin with the following lemma.

\begin{lemma}\label{L:Fact1}
Let $1\leq m\leq n$ be integers, and suppose $V\in \Gr^{> 0}(m,m+n)$. Then every $\mathcal{A}_{(m+n) \times m}$ with columns that span $V$ is of the form
\begin{align*}
\mathcal{A}=\mathcal{A}(Q,P,M) :=
Q
\begin{pmatrix}
P^T\begin{pmatrix}
\I_m\\
{\bf{0}}
\end{pmatrix}_{n\times m}\\
W_{0}
\end{pmatrix}
M
\end{align*}  
where $Q_{(m+n) \times (m+n)}$ is totally nonnegative, $M_{m\times m}$ is real nonsingular, and $P_{n\times n}$ is a permutation matrix such that $P^T\begin{pmatrix}
\I_m & {\bf{0}}
\end{pmatrix}^{T}_{n\times m}$ is totally nonnegative. 
\end{lemma}
\begin{proof}
All minors of $\mathcal{A}$ of maximal size have the same sign. Therefore, there exists an invertible $M_{m\times m}$ such that
$\mathcal{A}=\begin{pmatrix} B^{T} & W_{0}^{T} \end{pmatrix}^{T} M,$ and $B$ is ${n\times m}$ totally positive. We rewrite this factorization as 
\begin{align*}
\mathcal{A} = 
\begin{pmatrix}
\begin{pmatrix}
B_{n\times m} & {\bf 0} 
\end{pmatrix}_{n\times n} P_{n\times n} & {\bf 0}_{n\times m} \\
{\bf 0}_{m\times n} & \I_{m}
\end{pmatrix}_{(m+n) \times (m+n)}
\begin{pmatrix}
P_{n\times n}^T\begin{pmatrix}
\I_m \\ {\bf 0} 
\end{pmatrix}_{n\times m}\\
{W_{0}}
\end{pmatrix}_{(m+n)\times m}M_{m\times m}.
\end{align*}
As required $Q:=\begin{pmatrix}
\begin{pmatrix}
B & {\bf 0} 
\end{pmatrix}P & {\bf 0}_{n\times m} \\
{\bf 0}_{m\times n} & \I_{m}
\end{pmatrix}$ is totally nonnegative, $M$ is nonsingular, and $P$ is a permutation matrix such that $P^T\begin{pmatrix}\I_m & {\bf{0}}\end{pmatrix}^T_{n\times m}$ is totally nonnegative.
This completes the proof.
\end{proof}

It is a known fact that $\Gr^{> 0}(m,m+n)$ is dense in $\TGr$ (see for instance \cite{SV}). With this and Lemma~\ref{L:Fact1} in hand, next we obtain a dense subset of $\TGr$ which is nice enough to prove our main results stated in Section~\ref{ChevalleyOp}.

\begin{lemma}\label{L:Fact2}
Suppose $1\leq m\leq n$ are integers, and let $V\in \Gr^{> 0}(m,m+n).$ Then for every $\mathcal{A}$ with column vectors that span $V,$ there exists a nonsingular totally nonnegative matrix $T,$ diagonal matrices $D:=\diag(d_1,\dots,d_m)$ and $L:=\diag(\ell_1,\dots,\ell_m)$ with each $d_{k},\ell_{k}>0,$ and a real nonsingular $M$ such that
\begin{align*}
\mbox{the Frobenius norm}\quad \Bigg{\lVert} \mathcal{A}-
T
\begin{pmatrix}
P^T
\begin{pmatrix}
D\\
{\bf{0}}
\end{pmatrix}\\
L W_{0}
\end{pmatrix}
M \Bigg{\rVert}_{\mathcal{F}} \quad \mbox{is arbitrarily small}
\end{align*}  
where $P_{n\times n}$ is a permutation matrix such that $P^T\begin{pmatrix}\I_m & {\bf{0}} \end{pmatrix}^T_{n\times m}$ is totally nonnegative. In fact, we either can have $T=\I,$ or $T=\prod_{j=1}^{k} \big{(} \I + w_j \El_{(u_j,v_j)} \big{)}$ for some integer $k\geq 1$ with each $w_j > 0$ and $u_j\in \{v_j \pm 1\}.$
\end{lemma}
\begin{proof}
Consider the form of $\mathcal{A}:=\mathcal{A}(Q,P,M)$ in Lemma~\ref{L:Fact1} and, using Theorem~\ref{TN-classification}, pick an invertible totally nonnegative matrix $T'$ approximating $Q.$ This $T'=TD'$ where $D'$ is a positive definite diagonal matrix, and either $T=\I$ or $T=\prod_{j=1}^{k} \big{(} \I + w_j \El_{(u_j,v_j)} \big{)}$ for some integer $k\geq 1,$ with each $w_j > 0$ and $u_j\in \{v_j \pm 1\}.$ We obtain that the Frobenius norm
\begin{align*}
\Bigg{\lVert} \mathcal{A}-
T
\begin{pmatrix}
P^T
\begin{pmatrix}
D\\
{\bf{0}}
\end{pmatrix}\\
LW_{0}
\end{pmatrix}
M \Bigg{\rVert}_{\mathcal{F}} & = 
\Bigg{\lVert}
Q
\begin{pmatrix}
P^T
\begin{pmatrix}
\I_m \\ 
{\bf 0} 
\end{pmatrix}\\
{W_{0}}
\end{pmatrix} M - 
T'
\begin{pmatrix}
P^T
\begin{pmatrix}
\I_{m}\\
{\bf{0}}
\end{pmatrix}
\\
W_{0}
\end{pmatrix}
M \Bigg{\rVert}_{\mathcal{F}}\\
& \leq 
\big{\lVert} Q- T' \big{\rVert}_{\mathcal{F}} 
\Bigg{\lVert}
\begin{pmatrix}
P^T
\begin{pmatrix}
\I_m \\ 
{\bf 0} 
\end{pmatrix}
\\
{W_{0}}
\end{pmatrix} M \Bigg{\rVert}_{\mathcal{F}}
\mbox{ is arbitrarily small.}
\end{align*}
This completes the proof.
\end{proof}

\begin{remark}\label{test-matrices}
The projective geometry of the Grassmannian together with the ``density'' Lemma~\ref{L:Fact2} assert that it is sufficient to prove \eqref{quad_ineq} for all $V\in \TGr$ that are spanned by the columns of matrices
\begin{align*}
\mathcal{A}=\mathcal{A}(T,P,D,L):=T
\begin{pmatrix}
P^T
\begin{pmatrix}
D\\
{\bf{0}}
\end{pmatrix}\\
L W_{0}
\end{pmatrix}
\end{align*}
where $T=\prod_{j=1}^{k} \big{(} \I + w_j \El_{(u_j,v_j)} \big{)}$ for some integer $k\geq 0$ (assuming $T=\I$ when $k=0$) with each $w_j > 0$ and $u_j\in \{v_j \pm 1\},$ $D$ and $L$ positive definite diagonal matrices, and $P_{n\times n}$ is any permutation matrix such that $P^T
\begin{pmatrix}\I_{m} & {\bf{0}} \end{pmatrix}^T$ is totally nonnegative.
\end{remark}

We are now ready to prove our first main result.

\begin{proof}[Proof of Theorem~\ref{L:1}]
Suppose $u,v\in [m+n]$ are consecutive, and define 
\begin{align*}
F_{\In}(\mathcal{A}):=\sum_{\alpha\in \mathcal{I}} c_{\alpha} \Delta_{I_{\alpha}}(\mathcal{A}) \Delta_{J_{\alpha}}(\mathcal{A}) \quad \mbox{and}\quad 
G_{\In(u,v)}(\mathcal{A}):=\sum_{\alpha\in \mathcal{I}(u,v)} c_{\alpha} \Delta_{I_{\alpha}(u,v)}(\mathcal{A}) \Delta_{J_{\alpha}(u,v)}(\mathcal{A})
\end{align*}
for all test matrices $\mathcal{A}:=\mathcal{A}(T,P,D,L)$ in Remark~\ref{test-matrices}.\medskip 

\noindent $(\Longrightarrow)$ Suppose \eqref{quad_ineq} is valid. Then $F_{\In}(\mathcal{A})\geq 0$ for all $\mathcal{A}:=\mathcal{A}(T,P,D,L)$ in Remark~\ref{test-matrices}. Consider 
\[
\mathcal{B}:=\big{(}\I+w\El_{(u,v)}\big{)} \mathcal{A} \quad\mbox{for $w>0$ and $u,v\in [m+n]$ consecutive}
\]
such that $\So_{(u,v)}$ changes the parameters $(\In,I_\A,J_\A)$ for some $\A\in \In$; see Definition~\ref{defn:main:2}. Then $\mathcal{B}$ belongs to the test matrices in Remark~\ref{test-matrices}. Using the Cauchy--Binnet formula we compute the maximal minors of $\mathcal{B}$ in terms of maximal minors of $\mathcal{A}$:
\begin{align*}
\Delta_{I}(\mathcal{B})=
\begin{cases}
\Delta_{I}(\mathcal{A})+w\Delta_{I(u,v)}(\mathcal{A}) & \mbox{ if }u\in I \mbox{ and }v\not\in I, \mbox{ and}\\
\Delta_{I}(\mathcal{A}) & \mbox{ otherwise}. 
\end{cases}
\end{align*}
This gives
\begin{align*}
F_{\mathcal{I}}(\mathcal{B}) = 
\begin{cases}
F_{\mathcal{I}}(\mathcal{A}) + wG_{\mathcal{I}(u,v)}(\mathcal{A}) & \mbox{ if } \max_{\alpha \in \mathcal{I}}|\alpha(u,v)| = 1, \mbox{ and} \\
F_{\mathcal{I}}(\mathcal{A}) + \dots + w^2 G_{\mathcal{I}(u,v)}(\mathcal{A}) & \mbox{ if } \max_{\alpha \in \mathcal{I}}|\alpha(u,v)| = 2. \\
\end{cases}
\end{align*}
Since $F_{\mathcal{I}}(\mathcal{B})\geq 0$ for all $w>0,$ we must have $G_{\mathcal{I}(u,v)}(\mathcal{A})\geq 0$ for all test matrices $\mathcal{A}$ in Remark~\ref{test-matrices}. This completes the proof of the forward implication.\medskip

\noindent $(\Longleftarrow)$ It is not difficult to see that for all possibilities other than $m=n,$ $I_\A \cap J_\A=\emptyset,$ and $I_\A\cup J_\A = [2m],$ the converse follows from the forward implication itself. So we assume that $m=n,$ $I_\A \cap J_\A = \emptyset,$ and $I_\A \cup J_\A = [2m].$ This implies that the admissible permutation matrices $P$ in Remark~\ref{test-matrices} can only be the identity matrix. Therefore, we assume $G_{\In(u,v)}(\mathcal{A})\geq 0$ for all test matrices $\mathcal{A}:=\mathcal{A}(T,I,D,L)$ in Remark~\ref{test-matrices} and all consecutive $u,v\in [2m].$\medskip

\noindent\textbf{Claim 1:} If $F_{\In}(\mathcal{A}(I,I,D,L))\geq 0$ for all positive definite diagonal matrices $D$ and $L$, then $F_{\In}(\mathcal{A}) \geq 0$ for all test matrices $\mathcal{A}=\mathcal{A}(T,I,D,L)$ in Remark~\ref{test-matrices}.\medskip

\noindent\textit{Proof of Claim 1:} To prove that $F_{\In}(\mathcal{A})\geq 0$ for all test matrices in $\mathcal{A}:=\mathcal{A}(T,I,D,L)$ (using Theorem~\ref{TN-classification}) we induct on the number of factors in $T.$ We already have the required nonnegativity when $k=0,$ i.e. when $T=\I.$ For $k=1,$ suppose $\mathcal{B}= (\I + w\El_{(u,v)})\mathcal{A},$ where $\mathcal{A}=\mathcal{A}(I,I,D,L).$ Note that
\begin{align}\label{Eqn1:ProofofA}
F_{\mathcal{I}}(\mathcal{B}) = 
F_{\mathcal{I}}(\mathcal{A}) + wG_{\mathcal{I}(u,v)}(\mathcal{A}).
\end{align}
We have that $F_{\In}(\mathcal{A})$ and $G_{\mathcal{I}(u,v)}(\mathcal{A})$ are nonnegative, implying that $F_{\In}(\mathcal{B})\geq 0.$ Now suppose $k\geq 2$ and consider $\mathcal{B}:=\big{(} \I + w \El_{(u,v)}\big{)} \mathcal{A},$ where $\mathcal{A}:=\mathcal{A}(T,I,D,L)$ in which $T$ has $k-1$ factors. The inductive step allows one to assume that $F_{\In}(\mathcal{A}) \geq 0$ for these $\mathcal{A}:=\mathcal{A}(T,I,D,L).$ Using an argument similar to the one for the case in \eqref{Eqn1:ProofofA} we have $F_{\In}(\mathcal{A})\geq 0$ for all test matrices $\mathcal{A}$ in Remark~\ref{test-matrices}.\qed\medskip

\noindent\textbf{Claim 2:} $F_{\In}(\mathcal{A}(I,I,D,L))\geq 0$ for all positive definite diagonal matrices $D$ and $L.$\medskip

\noindent\textit{Proof of Claim 2:} Recall the composite operations in Definition~\ref{compositive-Chev-Op-1}. One can see that for a consecutive pair $u,v$ in $(I_\A \cup J_\A)\setminus (I_\A \cap J_\A),$ the application of $\So_{(u,v)}^{\mathcal{P}_1},$ where $\mathcal{P}_1:=(I_\A,J_\A)_{\A\in \In},$ on \eqref{quad_ineq} yields an inequality of the form \eqref{quad_ineq}, say for parameters $\mathcal{P}_2:=(I_\A^{(2)},J_\A^{(2)})_{\A\in \In^{(2)}}$ such that
\begin{align*}
|(I_\A^{(2)} \cup J_\A^{(2)})\setminus (I_\A^{(2)} \cap J_\A^{(2)})| = |(I_\A \cup J_\A)\setminus (I_\A \cap J_\A)|-2.
\end{align*}
This implies that after exactly $m$ successive applications of such $\So_{(u,v)}$ over \eqref{quad_ineq}, call it $\So,$ we obtain an inequality of the form \eqref{quad_ineq}, say for parameters $\In(\So),I_\A(\So),J_\A(\So)$ such that $|(I_\A(\So) \cup J_\A(\So))\setminus (I_\A(\So) \cap J_\A(\So))| = 0.$ This means that each $I_\A(\So) = J_\A(\So)$ and the corresponding inequality is equivalent to $\sum_{\A\in\In(\So)}c_\A \geq 0,$ involving just the coefficients $c_\A.$
The proof of Claim 2 completes once we prove the next two claims (below). Note that the scalars 
\begin{align*}
F_{\In}(\mathcal{A})=\sum_{\alpha\in \mathcal{I}} c_{\alpha} \Delta_{I_{\alpha}}(\mathcal{A}) \Delta_{J_{\alpha}}(\mathcal{A})\quad\mbox{for}\quad\mathcal{A}:=\mathcal{A}(I,I,D,L) \mbox{ in Remark~\ref{test-matrices}}   
\end{align*}
involve precisely those $I_\A$ and $J_\A$ that correspond to Pl\"ucker coordinates for the principal minors (Definition~\ref{defn:main:3}), as the other Pl\"ucker coordinates vanish. The next step is to recover these scalars via a composition of Chevalley operations. Consider the following composite operation:
\begin{align*}
\So_{P}:=\So_{(1,2m)}^{\mathcal{P}_{m}} \circ \dots \circ \So_{(\kappa,2m+1-\kappa)}^{\mathcal{P}_{m-\kappa+1}}\circ \dots \circ \So_{(m-1,m+2)}^{\mathcal{P}_2} \circ \So_{(m,m+1)}^{\mathcal{P}_1},
\end{align*}
where $\mathcal{P}_{m-\kappa+2}:=(I_\A^{(m-\kappa+2)},J_\A^{(m-\kappa+2)})_{\A\in\In^{(m-\kappa+2)}}$ is the parameter for the inequality obtained by operating $\So_{(\kappa,2m+1-\kappa)}^{\mathcal{P}_{m-\kappa+1}} \circ \dots \circ \So_{(m,m+1)}^{\mathcal{P}_1}$ on $\mathcal{P}_1=(I_\A,J_\A)_{\A\in\In},$ for $\kappa =m,\dots,2.$ As discussed above, the application of $\So_{P}$ yields \eqref{quad_ineq} which is of the form $\sum_{\A\in \In(\So_{P})}c_\A.$ Moreover, for this composition, each $\A\in \In(\So_{P})$ corresponds to $I_\A,J_\A$ that are principal minor Pl\"ucker coordinates (Definition~\ref{defn:main:3}). For completeness, we give a proof.\medskip

\noindent\textbf{Claim 2(a):} $\Delta_{I}(\mathcal{A})\neq 0$ for $\mathcal{A}=\mathcal{A}(I,I,D,L)$ if and only if whenever $\kappa \in I$ then $2m+1-\kappa \not \in I.$\medskip

\noindent\textit{Proof of Claim 2(a):} Note that rows $\kappa$ and $2m+1-\kappa$ are scalar multiples of each other for all $\mathcal{A}=\mathcal{A}(I,I,D,L).$ Therefore if $\kappa, 2m+1-\kappa \in I$ then $\Delta_{I}(\mathcal{A})=0.$ On the other hand, if whenever $\kappa \in I$ then $2m+1-\kappa \not \in I,$ then the submatrix $\mathcal{A}_{I,[m]}$ is a row permutation of a diagonal matrix with rows coming from the rows of $D$ and $L.$ And since $D,L$ have positive diagonals, $\Delta_{I}(\mathcal{A})=\det \mathcal{A}_{I,[m]}\neq 0.$\qed\medskip

\noindent\textbf{Claim 2(b):} $\alpha\in \mathcal{I}(\So_{P})$ if and only if whenever $\kappa\in \mathrm{X}$ then $2m+1-\kappa\not\in \mathrm{X}$ for $\mathrm{X}\in \{I_\alpha,J_{\alpha}\}.$\medskip

\noindent\textit{Proof of Claim 2(b):} By the construction of $\So_P$, $\alpha \in \mathcal{I}(\So_{P})$ if and only if $\{\kappa,2m+1-\kappa\}\cap \mathrm{X}\neq \emptyset$ for all $\kappa\in [m],$ for $\mathrm{X}\in \{I_\A,J_\A\}.$ Therefore $I_\A$ contains exactly one of $\kappa$ and $2m+1-\kappa$ for each $\kappa\in [m].$ The same is true for $J_\A$ with $\A\in \mathcal{\So_P},$ since $J_\A=[2m]\setminus I_\A.$\qed\medskip

Combining Claims~2(a) and 2(b), and using the forward implication of Theorem~\ref{L:1} (proved above), since $G_{\In(m,m+1)}\geq 0$ over $\TGr,$ we have $F_{\In}(\mathcal{A}(I,I,D,L))=\sum_{\A\in \In(\So_P)}c_\A \geq 0,$ proving Claim~2.\qed\medskip

Claim 2 along with Claim 1 completes the proof of the reverse implication.
\end{proof}

We prove the following lemma before proceeding with the proof of Theorem~\ref{L:2}.

\begin{lemma}\label{L:1:weakform}
Let $1\leq m \leq n$ be integers; notations in Definitions~\ref{defn:main:1} and \ref{defn:main:2}. Suppose $(u,v) \in \mathcal{S}$ for consecutive integers $u,v\in [m+n]$ if any one of the following holds for all $\A\in \In$:
\begin{enumerate}
    \item $u\in I_\A\cap J_\A.$
    \item $u\in I_\A$ and $u\not\in J_\A$ (or vice versa) and $v\not\in I_\A\cup J_\A.$
    \item $u\not\in I_\A\cup J_\A.$
\end{enumerate}
The inequality \eqref{quad_ineq} is valid if and only if for all $(u,v)\in \mathcal{S},$
\begin{align}\label{L:1:weakform:Eqn:2}
\sum_{\alpha\in \mathcal{I}(u,v)} c_{\alpha} \Delta_{I_{\alpha}(u,v)} \Delta_{J_{\alpha}(u,v)} \geq 0 \quad \mbox{over}\quad\TGr.
\end{align}
Moreover, $\mathcal{I}(u,v)=\In$ for all $(u,v)\in \mathcal{S}.$ 
\end{lemma}
\begin{proof}
This is obvious if $\mathcal{S}=\emptyset$; else proceed as in the forward implication of Theorem~\ref{L:1}.
\end{proof}

\begin{proof}[Proof of Theorem~\ref{L:2}]
For every action of $\So_{(u,v)}$ in this proof, we use Lemma~\ref{L:1:weakform}. Suppose \eqref{quad_ineq} holds. Let $\xi= |I_{\A}\cap J_{\A}|$ and $I_{\A}\cap J_{\A} = \{ r_1 < \dots < r_\xi\}.$ Since $m\leq n,$ the number of elements in $[m+n]\setminus (I_{\A}\cup J_{\A})$ is at least $\xi.$ Suppose $[m+n]\setminus (I_{\A}\cup J_{\A}) :=\{s_1>\dots>s_{\xi}>\dots\},$ and define for $j\in [\xi]$:
\begin{align*}
\So_{(r_j,j)}:= 
\begin{cases}
\So_{(j+1,j)} \circ \dots \circ \So_{(r_j,r_j-1)} & \mbox{if } r_j \neq j, \\
\So_{(r_j,r_j)} & \mbox{otherwise}.
\end{cases}
\end{align*}
We define another set of operations for all $j\in [\xi]$:
\begin{align*}
\So_{(s_j,m+n+1-j)}:= 
\begin{cases}
\So_{(m+n+1-j,m+n+1-j-1)} \circ \dots \circ \So_{(s_j+1,s_j)} & \mbox{if } s_j \neq m+n+1-j, \\
\So_{(s_j,s_j)} & \mbox{otherwise}.
\end{cases}
\end{align*}
Now apply the composite operation
$
\So_{(s_\xi,m+n+1-\xi)} \circ \dots \circ \So_{(s_1,m+n)} \circ \So_{(r_\xi,\xi)} \circ \dots \circ \So_{(r_1,1)}
$
on \eqref{quad_ineq}. The resultant inequality is of the form
\begin{align}\label{L:2:Eqn3}
\sum_{\alpha\in \mathcal{I}} c_{\alpha} \Delta_{I_{\alpha}'} \Delta_{J_{\alpha}'} \geq 0 \quad\mbox{over}\quad \TGr.
\end{align}
where $I_{\A}'\cap J_{\A}'=[\xi],$ and $[m+n+1-\xi,m+n]\subseteq [m+n]\setminus (I_{\A}'\cup J_{\A}').$ Let $\omega := 2\eta - m +\xi,$ and suppose $(I_{\A}'\setminus J_{\A}') \cup (J_{\A}'\setminus I_{\A}')=\{t_1>\dots>t_{m-\xi} > u_{\omega} > \dots > u_1\}.$ We define the following operations for $j\in [m-\xi]$:
\begin{align*}
\So_{(t_j,m+n+1-j-\xi)}:= 
\begin{cases}
\So_{(m+n-j-\xi,m+n+1-j-\xi)} \circ \dots \circ \So_{(t_j,t_j+1)} & \mbox{ if } t_j \neq m+n+1-j-\xi, \\
\So_{(t_j,t_j)} & \mbox{ otherwise}.
\end{cases}
\end{align*}
We define the following operations for $j\in [\omega]$:
\begin{align*}
\So_{(u_j,\xi+j)}:= 
\begin{cases}
\So_{(\xi+j-1,\xi+j)} \circ \dots \circ \So_{(u_j,u_j-1)} & \mbox{ if } u_j \neq \xi+j, \\
\So_{(u_j,u_j)} & \mbox{ otherwise}.
\end{cases}
\end{align*}
Now apply the composite operation 
$
\So_{(u_{\omega},\xi+\omega)}\circ \dots \circ \So_{(u_1,\xi+1)}\circ \So_{(t_{m-\xi},n+1)}\circ \dots \circ \So_{(t_1,m+n-\xi)}
$
on inequality \eqref{L:2:Eqn3}. The resultant inequality is of the form:
\begin{align}\label{L:2:Eqn4}
\sum_{\alpha\in \mathcal{I}} c_{\alpha} \Delta_{I_{\alpha}''} \Delta_{J_{\alpha}''} \geq 0 \quad\mbox{over}\quad \TGr.
\end{align}
where $I_{\A}''\cap J_{\A}''=[\xi],$ $[m+n+1-\xi,m+n]\subseteq [m+n]\setminus (I_{\A}''\cup J_{\A}''),$ and $(I_{\A}''\setminus J_{\A}'')\cup (J_{\A}''\setminus I_{\A}'') = [\xi+1,\xi+\omega]\cup [n+1,m+n-\xi].$
We now re-write inequality \eqref{L:2:Eqn4} in the form of the first inequality in Theorem~\ref{Ineq-equivalence}:
\begin{align}\label{L:2:Eqn5}
\sum_{P_1,Q_1,P_2,Q_2} c_{P_1,Q_1,P_2,Q_2}\det A_{P_1,Q_1} \det A_{P_2,Q_2} & \geq 0\quad \forall A_{n\times m} \quad TNN.
\end{align}
Note that the indices in $I_{\A}''$ and $J_{\A}''$ imply that inequality \eqref{L:2:Eqn5} is only for the leading $m\times m$ principal submatrices of $A_{n\times m}.$ Therefore \eqref{L:2:Eqn5} simplifies to
\begin{align}\label{L:2:Eqn6}
\sum_{P_1,Q_1,P_2,Q_2} c_{P_1,Q_1,P_2,Q_2}\det A_{P_1,Q_1} \det A_{P_2,Q_2} & \geq 0\quad \forall A_{m\times m} \quad TNN.
\end{align}
Also, note that $[\xi] = P_1\cap Q_1\cap P_2 \cap Q_2.$ Therefore, using Theorem~\ref{TN-classification} for the diagonal matrix $D=\diag(d_1,\dots,d_m)$ is exactly in the middle of the factorization, inequality \eqref{L:2:Eqn6} changes to the following:
\begin{align}\label{L:2:Eqn7}
\sum_{P_1,Q_1,P_2,Q_2} c_{P_1,Q_1,P_2,Q_2} \prod_{j=1}^{\xi} d_j^2 \det A_{P_1\setminus [\xi],Q_1\setminus [\xi]} \det A_{P_2\setminus [\xi],Q_2\setminus [\xi]} & \geq 0\quad \forall A_{m\times m} \quad TNN.
\end{align}
It is evident that the above inequality is only for the lagging principal submatrix $A_{[m]\setminus [\xi],[m]\setminus [\xi]}$. So, deleting the first $\xi$ rows/columns, and renumbering the rows/columns simplifies the above inequality further to:
\begin{align}\label{L:2:Eqn8}
\sum_{P_1',Q_1',P_2',Q_2'} c_{P_1',Q_1',P_2',Q_2'} \det A_{P_1',Q_1'} \det A_{P_2',Q_2'} & \geq 0\quad \forall A_{\eta \times \eta} \quad TNN.
\end{align}
The third equivalent version formulated in Theorem~\ref{Ineq-equivalence} of inequality \eqref{L:2:Eqn8} is the required inequality \eqref{L:2:Eqn2}. For the converse: note that each of the steps followed above is reversible without affecting the validity of the inequalities. Performing the reverse of these steps would yield inequality \eqref{quad_ineq} starting from \eqref{L:2:Eqn2}. This completes the proof.
\end{proof}

We are now ready to write the final proof of this paper.

\begin{proof}[Proof of Theorem~\ref{L:3}]
Theorem~\ref{L:1} implies that if inequality~\eqref{L:3:Eqn:1} is valid then so is \eqref{L:3:Eqn:2}. For the converse, suppose \eqref{L:3:Eqn:1} does not hold, but \eqref{L:3:Eqn:2} is valid for all consecutive $u,v\in [n+1].$ Suppose $u,v\in [2n]$ are consecutive, and define 
\begin{align*}
F_{\In}(\mathcal{A}):=\sum_{\alpha\in \mathcal{I}} c_{\alpha} \Delta_{I_{\alpha}}(\mathcal{A}) \Delta_{J_{\alpha}}(\mathcal{A}) \quad \mbox{and}\quad 
G_{\In(u,v)}(\mathcal{A}):=\sum_{\alpha\in \mathcal{I}(u,v)} c_{\alpha} \Delta_{I_{\alpha}(u,v)}(\mathcal{A}) \Delta_{J_{\alpha}(u,v)}(\mathcal{A})
\end{align*}
for all test matrices $\mathcal{A}:=\mathcal{A}(T,P,D,L)$ in Remark~\ref{test-matrices}. We claim that $G_{\In(u,v)}(\mathcal{A})\geq 0$ for all test matrices $\mathcal{A}$ in Remark~\ref{test-matrices}, for all consecutive $u,v\in [2n].$ Then $\{u,v\}\neq\{n,n+1\}$ as inequality \eqref{L:3:Eqn:2} is valid for $u=n,v=n+1.$

Suppose $\mathcal{P}_1:= (I_\A,J_\A)_{\A\in\In}$ (as in Definition~\ref{defn:main:1}). Recall the composition of Chevalley operations in Definition~\ref{compositive-Chev-Op-1}. One can see that for a consecutive pair $u,v$ in $(I_\A \cup J_\A)\setminus (I_\A \cap J_\A),$ the application of ${\So}_{(u,v)}^{\mathcal{P}_1}$ on \eqref{quad_ineq} yields an inequality of the form \eqref{quad_ineq}, say for parameters $(I_\A^{(2)},J_\A^{(2)})_{\A\in \In^{(2)}},$ which satisfy 
\begin{align*}
|(I_\A^{(2)} \cup J_\A^{(2)})\setminus (I_\A^{(2)} \cap J_\A^{(2)})| = |(I_\A \cup J_\A)\setminus (I_\A \cap J_\A)|-2.
\end{align*}
Suppose $P:=\{\{u_j,v_j\}\}$ is a partition of $(I_\A \cup J_\A)\setminus (I_\A \cap J_\A)$ compatible with some $321$-avoiding permutation. Corresponding to each of these partitions, define
\[
\So_{P}:=\So_{(u_n,v_n)}^{\mathcal{P}_n} \circ \dots \circ \So_{(u_1,v_1)}^{\mathcal{P}_1},
\]
where $\mathcal{P}_k:=(I_\A^{(k)},J_\A^{(k)})_{\A\in\In^{(k)}}$ is the parameter for the inequality obtained by operating $\So_{(u_{k-1},v_{k-1})}^{\mathcal{P}_{k-1}} \circ \dots \circ \So_{(u_1,v_1)}^{\mathcal{P}_1}$ on $\mathcal{P}_1=(I_\A,J_\A)_{\A\in\In},$ for $k=2,\dots,\eta.$ If \eqref{L:3:Eqn:1} is not valid, then Theorems~\ref{L:1} and \ref{L:2} imply that there exists a partition $P:=\{\{u_j,v_j\}\}$ of $(I_\A \cup J_\A)\setminus (I_\A \cap J_\A),$ compatible with some 321-avoiding permutation, such that the action of $\So_{P}$ over \eqref{L:3:Eqn:1} yields a negative scalar $\sum_{\A\in \In(P)}c_\A$. (For clarity: here we only adapt an argument from the proof of Theorem~\ref{Main-thm}, and do not use Theorem~\ref{Main-thm} itself; so there is no circular argument.) \medskip

\noindent\textbf{Case 1:} $\{u_\ell,v_\ell\}=\{u_1^{*},v_1^{*}\}$ for some $\ell.$ In this case define a partition $P':=\{\{w_j,x_j\}\}$ where each $(w_j,x_j)=(u_j,v_j)$ for all $j\not\in \{2,\ell\},$ $(w_2,x_2)=(u_\ell,v_\ell),$ and $(w_\ell,x_\ell)=(u_2,v_2).$ Suppose $Q'=P'\setminus\{\{w_1,x_1\},\{w_2,x_2\}\}.$ One can see that $Q'$ is a partition for \eqref{L:3:Eqn:2} for $(u,v)=(u_1,v_1)$ that is compatible with some 321-avoiding permutation. In fact, it is not difficult to see that 
\begin{align*}
\sum_{\mathcal{J}(Q')}c_\A=\sum_{\In(P')}c_\A=\sum_{\In(P)}c_\A<0,
\end{align*}
where $\mathcal{J}=\In(u_1,v_1).$ In other words, it can be seen that the sequence of Chevalley operations induced by $Q'$ operates on \eqref{L:3:Eqn:2} for $(u,v)=(u_1,v_1).$ Since \eqref{L:3:Eqn:2} is valid, the aformention sum can not be negative. Therefore we have a contradiction.\medskip

\noindent\textbf{Case 2:} There exist $p,q$ such that $\{u_p,v_p\},\{u_q,v_q\}\in P,$ where $(u_p,v_p)=(u_1^*,v_p)$ and $(u_q,v_q)=(v_1^*,v_q).$ We construct partition $P'$. Suppose $r=\max\{p,q\}.$ Then the first $r-1$ elements $\{u_j,v_j\}$ are in $P';$ add $\{u_p,u_q\},\{v_p,v_q\}$ as the next two elements, and finally add the remaining element from $P\setminus \{\{u_p,v_p\},\{u_q,v_q\}\}.$ Since \eqref{L:3:Eqn:1} involves only the principal minor Pl\"ucker coordinates, we have
\begin{align*}
\sum_{\In(P')}c_\A=\sum_{\In(P)}c_\A.
\end{align*}
Furthermore, note that the partition $P'$ falls in the category of Case 1 above. Therefore we follow the steps in Case 1 for $P'$ to arrive at a contradiction, and conclude this proof.
\end{proof}

\section{Future work}\label{futurework} 

\noindent We present a few questions for future work.
\begin{enumerate}
\item Rhoades--Skandera showed that the dimension of the vector space in \eqref{eq:tlspace} is the $n$-th Catalan number $C_n=\frac{1}{n+1}\binom{2n}{n}$ and obtained a basis made up of the Temperley--Lieb immanants. This dimension aligns with the Catalan many nonnegativity conditions in Theorem~\ref{Main-thm}. Therefore in light of Theorem~\ref{TLnew}, it would be interesting to see if the dimension of the subspace given by $$\spn_{\mathbb R} \big{\{} \det {\bf x}_{P,P}  \det {\bf x}_{P^c,P^c}  \,:\, P \subseteq [n] \big{\}}$$
is $\binom{n}{\floor{n/2}}$, and to obtain a basis for it.

\item As we mentioned in Remark~\ref{recall-Ruv}, the idea of Chevalley operations originates in the cluster algebraic structure of the Grassmannian. Therefore, it would be interesting to see if the main results in this work can be proved using the cluster algebraic structure of the TNN Grassmannian (instead of its combinatorial structure).

\item We employ a certain derivative (Chevalley operations) of the Chevalley generators of TNN matrices to obtain the characterization of a subclass of the TNN polynomials, precisely the corresponding dual canonical basis elements (Lusztig \cite{lusztig1998introduction}). Considering that the Chevalley generators are crucial for the totally nonnegative part of a reductive group (Lusztig \cite{LusztigTP}), it would be interesting to see if other (or more general) dual canonical basis elements can also be characterized via some analogous operation coming out of the corresponding Chevalley generators.
\end{enumerate}

\section*{Acknowledgments}
I am thankful to Melissa Sherman-Bennett for the discussions on TNN Grassmannians during an AIM Workshop on Total Positivity in Jul 2023. I also acknowledge the hospitality and stimulating environment provided by the American Institute of Mathematics (CalTech Campus) during this workshop. I am grateful to Misha Gekhtman for exposing me to the relevant literature and for many valuable conversations, which led me to discover the cluster algebraic origin of Chevalley operations. All these discussions took place in India during Oct-Nov 2024 at a workshop on Discrete Integrable Systems held at the International Centre for Theoretical Sciences (ICTS) in Bangalore. I am thankful to the ICTS for the hospitality and stimulating environment during the workshop. Finally, I would like to express my sincere gratitude to Apoorva Khare for his invaluable feedback and suggestions during the preparation of this article.

The author is supported by the Centre de recherches math\'ematiques and Laval (CRM-Laval) Postdoctoral Fellowship, and acknowledges the Pacific Institute for the Mathematical Sciences (PIMS), a SwarnaJayanti Fellowship from DST and SERB, Govt. of India (PI: Apoorva Khare), travel support from SPARC (Scheme for Promotion of Academic and Research Collaboration, MHRD, Govt. of India; PI: Tirthankar Bhattacharyya), and the University of Plymouth (UK) for the office space during a visit.


\vspace*{1cm}

\end{document}